\documentclass{article}[12pt]

\usepackage{makeidx}

\usepackage[english]{babel}
\usepackage[latin1]{inputenc}
\usepackage[T1]{fontenc}
\usepackage{amsmath, amsthm}
\usepackage{amscd}
\usepackage{amssymb}
\usepackage{amssymb}
\usepackage{latexsym}
\usepackage{amsmath}
\usepackage{amsfonts}
\usepackage{amsthm}

\usepackage{url}

\usepackage{color}

\usepackage{enumerate}
\usepackage{graphicx}
\usepackage{float}

\usepackage{lipsum}
\usepackage{tikz}
\usetikzlibrary{matrix,arrows,decorations.pathmorphing}
\usepackage[english]{babel}
\usepackage[latin1]{inputenc}
\usepackage[T1]{fontenc}
\usepackage{amsmath}
\usepackage{amssymb}
\usepackage{amsthm}
\usepackage{enumerate}

\newtheorem{thm}{Theorem}[section]
\newtheorem{theorem}[thm]{Theorem}

\newtheorem{corollary}[thm]{Corollary}

\newtheorem{proposition}[thm]{Proposition}

\newtheorem{remark}[thm]{Remark}

\theoremstyle{definition}
\newtheorem{example}[thm]{Example}

\makeatletter
\newcommand{\subalign}[1]{%
  \vcenter{%
    \Let@ \restore@math@cr \default@tag
    \baselineskip\fontdimen10 \scriptfont\tw@
    \advance\baselineskip\fontdimen12 \scriptfont\tw@
    \lineskip\thr@@\fontdimen8 \scriptfont\thr@@
    \lineskiplimit\lineskip
    \ialign{\hfil$\m@th\scriptstyle##$&$\m@th\scriptstyle{}##$\crcr
      #1\crcr
    }%
  }
}
\makeatother

\begin{document}

\centerline{\Large \bf Young-Capelli bitableaux, }

\bigskip

\centerline{\Large \bf   Capelli immanants in $\textbf{U}(gl(n))$}

\bigskip

\centerline{\Large \bf and }

\bigskip

\centerline{\Large \bf  the Okounkov quantum immanants}

\bigskip

\centerline{A. Brini and A. Teolis}
\centerline{\it $^\flat$ Dipartimento di Matematica, Universit\`{a} di
Bologna }
 \centerline{\it Piazza di Porta S. Donato, 5. 40126 Bologna. Italy.}
\centerline{\footnotesize e-mail of corresponding author: andrea.brini@unibo.it}
\medskip

\begin{abstract}

The set of standard Capelli bitableaux and the set of standard Young-Capelli bitableaux are bases
of $\mathbf{U}(gl(n))$, whose action on the Gordan-Capelli basis of polynomial algebra ${\mathbb C}[M_{n,n}]$
have remarkable properties (see, e.g. \cite{Brini1-BR}, \cite{Brini2-BR}, \cite{Brini3-BR}, \cite{Brini4-BR}).

We introduce a new class of elements of $\mathbf{U}(gl(n))$, called \emph{Capelli immanants},
that  can be efficiently computed and provide a system of linear generators of $\mathbf{U}(gl(n))$.
The Okounkov quantum immanants \cite{Okounkov-BR}, \cite{Okounkov1-BR} - \emph{quantum immanants}, for short - are proved
to be simple linear combinations of \emph{diagonal} Capelli immanants,
with explicit coefficients (Theorem \ref{quantum}, Eq. (\ref{quantum first})).
Quantum immanants can  also be expressed as sums of \emph{double Young-Capelli bitableaux}
(Theorem \ref{quantum second}, Eq. (\ref{schur second eq})). Since double Young-Capelli bitableaux
uniquely expand into linear combinations of \emph{standard} Young-Capelli bitableaux, Eq. (\ref{schur second eq})
leads to \emph{canonical} presentations of quantum immanants, and, furthermore,  it
 doesn't involve the computation of the irreducible characters of symmetric groups.

\end{abstract}


\tableofcontents


\section{Introduction}

The study of central elements  in $\mathbf{U}(gl(n))$ is a classical subject of the  theory of Lie algebras, see e.g. \cite{DIX-BR};
it is an old and actual one, since it may be regarded as
 an  offspring of the celebrated Capelli identity (\cite{Cap1-BR}, \cite{Cap4-BR}, \cite{Howe-BR}, \cite{HU-BR},
\cite{Procesi-BR}, \cite{Umeda-BR}, \cite{Weyl-BR}),
relates to its modern generalizations and applications (\cite{ABP-BR}, \cite{KostantSahi1-BR},
\cite{KostantSahi2-BR}, \cite{MolevNazarov-BR}, \cite{Nazarov-BR}, \cite{Okounkov-BR}, \cite{Okounkov1-BR},
 \cite{UmedaCent-BR})
as well as to the theory of {\it{Yangians}} (see, e.g.  \cite{Molev1-BR}, \cite{Molev-BR},  \cite{Nazarov2-BR}).

The center $\boldsymbol{\zeta}(n)$ of $\mathbf{U}(gl(n))$ is isomorphic to the algebra $\Lambda^*(n)$ of shifted symmetric polynomials
(e.g., {\it{factorial symmetric functions}}, \cite{BL1-BR},  \cite{CL-BR},  \cite{GH-BR})
via the Harish-Chandra isomorphism $\chi_n$ (see, e.g. \cite{OkOlsh-BR}).
The algebra $\Lambda^*(n)$ admits a quite relevant linear basis, the basis of the \emph{shifted Schur polynomials}
$s*_{\mu|n}$, $\widetilde{\mu}_1 \leq n$, discovered by Sahi \cite {Sahi1-BR},
and extensively studied by Okounkov and Olshanski \cite{OkOlsh-BR}.
\emph{Quantum immanants} are the preimages in $\boldsymbol{\zeta}(n)$
of the shifted Schur polynomials in  $\Lambda^*(n)$ \cite{Okounkov-BR}, \cite{Okounkov1-BR} (see also \cite{Nazarov2-BR}).

We define two set of linear generators in the enveloping algebra $\mathbf{U}(gl(n))$: the set of Young-Capelli bitableaux
and the set of Capelli immanants.
These two sets are obtained as the images of the corresponding  set of  generators in the polynomial algebra ${\mathbb C}[M_{n,n}]$,
 under the so-called
\emph{bitableaux correspondence isomomorphism}, which is an earlier result of the present authors \cite{Brini3-BR}, \cite{Brini4-BR}.

The action of the Capelli immanants on ${\mathbb C}[M_{n,n}]$ can be computed explicitly by using the method
of \emph{virtual variables} (Proposition \ref{column differential} below). Using this computation, we are able to express Okounkov's quantum immanants as  linear combinations of \emph{diagonal} Capelli immanants with explicit coefficients.

Our method heavily relies upon the
``Bitableax correspondence isomorphism/
 Koszul map'' Theorem  (BCK Theorem, for short) \cite{Brini4-BR}
that describes a pair of mutually inverse vector space isomorphisms,
the {\textit{Koszul map}} (\cite{Koszul-BR},  see also \cite{BriUMI-BR} and \cite{BriniTeolisKosz-BR})
$$
\mathcal{K} : \mathbf{U}(gl(n)) \rightarrow {\mathbb C}[x_{ij}] \cong \mathbf{Sym}(gl(n)),
$$
and the {\textit{bitableaux correspondence isomorphism}} (\cite{Brini3-BR}, \cite{Brini4-BR})
$$
\mathcal{K}^{-1} : {\mathbb C}[x_{ij}] \cong \mathbf{Sym}(gl(n)) \rightarrow \mathbf{U}(gl(n)),
$$
that deeply link the enveloping algebra $\mathbf{U}(gl(n))$ of the general linear Lie algebra $gl(n)$ and
the polynomial algebra ${\mathbb C}[M_{n,n}]$ of polynomials in the entries of a ``generic''
square matrix of order $n$. The BCK Theorem has to be regarded as a  sharpened
version of the PBW Theorem for the enveloping algebra $\mathbf{U}(gl(n))$.

The isomorphism $\mathcal{K}^{-1}$ maps a {\textit{(determinantal) bitableau}} $(S|T)$ in ${\mathbb C}[M_{n,n}]$ to the
{\textit{Capelli bitableau}} $[S|T]$ in $\mathbf{U}(gl(n))$
(\cite{Brini3-BR}, \cite{Brini4-BR}, \cite{Bri-BR}; see Section \ref{Capbit} below and Theorem \ref{KBT corr}).
Since the {\textit{standard}}  bitableaux are a basis
of ${\mathbb C}[M_{n,n}]$ (\cite{drs-BR}, \cite{DKR-BR}, \cite{DEP-BR}, \cite{rota-BR};
see subsection \ref{straight} below, Theorem \ref{theorem: standard basis}),
then the {\textit{standard}} Capelli bitableaux
are a  basis of $\mathbf{U}(gl(n))$ \cite{Brini3-BR}.

In the  polynomial algebra ${\mathbb C}[M_{n,n}]$,  {\it{column bitableaux}} are, up to a sign,
{\it{monomials}}. Their images in $\mathbf{U}(gl(n))$ - under  the isomorphism $\mathcal{K}^{-1}$ -
are the {\it{column Capelli bitableaux}} (Section \ref{column Capelli sec} below).

Therefore, column Capelli bitableaux play the same crucial role in $\mathbf{U}(gl(n))$ that monomials play
in ${\mathbb C}[M_{n,n}]$.
Capelli bitableaux and Young-Capelli bitableaux
expand  - up to  a global sign - into column Capelli bitableaux
just in the same way as
bitableaux, right symmetrized bitableaux and immanants expand into the  corresponding monomials
in  ${\mathbb C}[M_{n,n}]$.

The  expressions of column Capelli bitableaux  in $\mathbf{U}(gl(n))$ can be simply computed
(Proposition \ref{column Cap dev} below).
Furthermore, column Capelli bitableaux admit an  elegant and meaningful interpretation as
polynomial differential operators
in the {\textit{Weyl algebra}} associated to
the polynomial algebra ${\mathbb C}[M_{n,d}]$ (Proposition \ref{column differential} below).

The isomorphism $\mathcal{K}^{-1}$ leads to a  natural definition of the \emph{Capelli immanants}
$$
Cimm_{\lambda}[i_1 i_2 \cdots i_h;j_1 j_2 \cdots j_h], \quad \lambda \vdash h
$$
in $\mathbf{U}(gl(n))$ as images under $\mathcal{K}^{-1}$ of the classical \emph{immanants}
$$
imm_{\lambda}(i_1 i_2 \cdots i_h;j_1 j_2 \cdots j_h),
\quad
(i_1 i_2 \cdots i_h), (j_1 j_2 \cdots j_h) \in \underline{n}^h
$$
in the polynomial algebra ${\mathbb C}[M_{n,n}]$
(Littlewood and Richardson \cite{Little1-BR},  see also \cite{Little2-BR},  \cite{goulden-BR}).
Capelli immanants are generalizations of the famous Capelli determinant in $\mathbf{U}(gl(n))$,
just as immanants are generalizations of the determinant in ${\mathbb C}[M_{n,n}]$.

The isomorphism $\mathcal{K}^{-1}$ maps a {\textit{right symmetrized bitableau}} $(S|\fbox{$T$})$ in ${\mathbb C}[M_{n,n}]$
to the {\textit{Young-Capelli bitableau}} $[S|\fbox{$T$}]$ in $\mathbf{U}(gl(n))$
(Section \ref{Capbit} below, and Theorem \ref{image symm}).
Since the standard right symmetrized bitableaux $(S|\fbox{$T$})$ are the {\textit{Gordan-Capelli basis}}
of ${\mathbb C}[M_{n,n}]$ (\cite{Wall-BR}, \cite{Brini1-BR}, \cite{Bri-BR};
see Subsection \ref{symm bit} below, Theorem \ref{GC basis}),
then the standard Young-Capelli bitableaux $[S|\fbox{$T$}]$ are a basis of
$\mathbf{U}(gl(n))$.

Right symmetrized bitableaux $(S|\fbox{$T$})$ of shape $\lambda \vdash h$ expand into  {\textit{immanants}}
defined by the irreducible character $\chi^{\lambda}$  of the symmetric group $\mathbf{S}_h$
associated to the {\textit{same}}  shape $\lambda \vdash h$, and \emph{viceversa} (Propositions \ref{bit vs imm} and \ref{imm vs bit}).
Then, by applying the operator $\mathcal{K}^{-1}$, we obtain that
any Capelli immanant
$$Cimm_{\lambda}[i_1 i_2 \cdots i_h;j_1 j_2 \cdots j_h] $$
can be written as  a linear combination of standard Young-Capelli bitableaux $[U| \fbox{$V$}]$
in $\mathbf{U}(gl(n))$ of the \textbf{same} shape $\lambda$ and \emph{viceversa} (Theorems \ref{Cimm vs YC} and \ref{Cbit vs Cimm} below).

{\textit{Quantum immanants}}   (\cite{Okounkov-BR}, \cite{Okounkov1-BR})
 are proved  to be simple linear combinations of \emph{diagonal} Capelli immanants
 with explicit coefficients (Theorem \ref{quantum}, Eq. (\ref{quantum first})).
This Theorem, in combination with Proposition \ref{column Cap dev},
allows the computation of quantum immanants  to be reduced to a fairly simple process
(see, e.g. Example \ref{Schur example} below).

Quantum immanants can  also be expressed as sums of \emph{double Young-Capelli bitableaux}
(Theorem \ref{quantum second}, Eq. (\ref{schur second eq})). Since double Young-Capelli bitableaux
uniquely expand into linear combinations of \emph{standard} Young-Capelli bitableaux, Eq. (\ref{schur second eq})
leads to \emph{canonical} presentations of quantum immanants, and it
 doesn't involve the irreducible characters of symmetric groups.
Furthermore, Eq.  (\ref{schur second eq}) is better suited to the study of the eigenvalues on irreducible $gl(n)-$modules,
and of the duality in the algebra $\boldsymbol{\zeta}(n)$ (see our preliminary manuscript \cite{BriniTeolis-BR}, Section 4).

\section{The polynomial algebra ${\mathbb C}[M_{n,d}]$}\label{sec 2}

\subsection{Biproducts in ${\mathbb C}[M_{n,d}]$}

As usual, the \textit{algebra of algebraic forms in $n$ vector variables
of dimension $d$} is the polynomial algebra in $n \times d$ (commutative) variables:
$$
{\mathbb C}[M_{n,d}] =    {\mathbb C}[x_{ij}]_{i=1,\ldots,n; j=1,\ldots,d},
$$
and $M_{n,d}$ denotes the  matrix
with $n$ rows and $d$ columns with ``generic" entries $x_{ij}$:
$$
M_{n,d} = \left[ x_{ij} \right]_{i=1,\ldots,n; j=1,\ldots,d}=
 \left[
 \begin{array}{ccc}
 x_{11} & \ldots & x_{1d} \\
 x_{21} & \ldots & x_{2d} \\
 \vdots  &        & \vdots \\
 x_{n1} & \ldots & x_{nd} \\
 \end{array}
 \right].
$$

For the sake of readability, we will write $(i|j)$ in place of $x_{ij},$ and call
the alphabets $L = \{1, 2, \ldots, n \}$ and $P = \{1, 2, \ldots, d \}$
the {\textit{letter}} and the {\textit{place}} alphabets, respectively;
sometimes, we will consistently write ${\mathbb C}[ (i|j) ]_{i = 1, 2, \ldots, n; \ j = 1, 2, \ldots, d}$ in place of  ${\mathbb C}[M_{n,d}]$.

Let $\omega = i_1i_2 \cdots i_p$ be a word on the alphabet $L = \{1, 2, \ldots, n \}$,
and $\varpi = j_1j_1 \cdots j_q$ a word
on the alphabet
$P = \{1, 2, \ldots, d \}$.

Following \cite{rota-BR} and \cite{Brini1-BR},
the {\it{biproduct}} of the two words $\omega$ and $\varpi$
\begin{equation}\label{biproduct}
(\omega|\varpi) = (i_1i_2 \cdots i_p|j_1j_2 \cdots j_q)
\end{equation}
is the element of  ${\mathbb C}[M_{n,d}]$ defined in the following way:

\begin{itemize}

\item [--] If $p = q$, the biproduct $(\omega|\varpi)$
is the \emph{signed minor}
$$
(\omega|\varpi) = (-1)^{{p} \choose {2}} \ det \Big( \ (i_r|j_s) \ \Big)_{r, s = 1, 2, \ldots, p} \in {\mathbb C}[M_{n,d}].
$$

\item [--]
If $p \neq q$, the biproduct $(\omega|\varpi)$ is set to be zero.
\end{itemize}

\subsection{Bitableaux in ${\mathbb C}[M_{n,d}]$}

\subsubsection{Young tableaux}\label{comb Young tab}

 Let $\lambda \vdash h$ be a partition, and label the boxes of its Ferrers diagram
with the numbers $1, 2, \ldots , h$ in the following way:
$$
\begin{array}{lllll}
1 & 2 &  \cdots & \cdots & \lambda_1 \\
\lambda_1 + 1 &  \lambda_1 + 2 & \cdots & \lambda_1 + \lambda_2 &    \\
\cdots & \cdots & \cdots &  &  \\
\cdots & \cdots & h &  &  \\
\end{array}.
$$
A {\textit{Young tableau}} $T$ of shape $\lambda$ over a (finite) alphabet $\mathcal{A}$ is
a map $T : \underline{h} = \{1, 2, \ldots , h \}  \rightarrow \mathcal{A}$; the element $T(i)$
is the symbol in the cell $i$ of the tableau $T$.

The sequences
$$
\begin{array}{l}
T(1)  T(2) \cdots  T(\lambda_1),
\\
T(\lambda_1 + 1)  T(\lambda_1 + 2)  \cdots  T(\lambda_1 + \lambda_2),
\\
 \ldots \ldots
\end{array}
$$
are called the {\textit{row words}} of the Young tableau $T$.

We will also denote a Young tableau
by its sequence of rows words, that is $T = (\omega_1, \omega_2, \ldots, \omega_p)$.
Furthermore, the {\textit{word of the tableau}} $T$ is the concatenation
\begin{equation}\label{word}
w(T) = \omega_1\omega_2 \cdots \omega_p.
\end{equation}

The {\textit{content}} of a tableau $T$ is the function $c_T : \mathcal{A} \rightarrow \mathbb{N}$,
$$
c_T(a) = \sharp \{i \in \underline{h}; \ T(i) = a \}.
$$

A Young tableau $\mathbf{T}$ is said to be {\textit{multilinear}} if   $ \mathcal{A} = \underline{h} $ and the map
$\mathbf{T}$ is a permutation of $\underline{h}$. As usual, $\mathbf{T}^{-1}$ denotes the inverse map.
In the sequel, multilinear Young tableaux will be always denoted  by  bold symbols,
and $\mathbf{T}_0$ will denote the ``identity'' tableau $\mathbf{T}_0(i) = i$, $i = 1, 2, \ldots, h.$

Note that, given any Young tableau $S$ on an alphabet $\mathcal{A}$ and any multilinear Young tableau
$\mathbf{T}$ on the alphabet $\underline{h}$ of the same shape $\lambda \vdash h$, there exists a unique (specialization)  map
$J :  \underline{h}   \rightarrow \mathcal{A}$ such that
$$
S = J \circ {\mathbf{T}},
$$
that is
$
S(i) = (J \circ {\mathbf{T}})(i),  \  i = 1, 2, \ldots, h.
$

To stress this relation between $S$ and $\mathbf{T}$, we  write
\begin{equation}\label{specialization}
S = J_{\mathbf{T}}.
\end{equation}

Given a linear order on the  alphabet $\mathcal{A}$, a Young tableau over $\mathcal{A}$ is said to be
{\textit{(semi)standard}} whenever its rows are increasing from left to right and its columns are non-decreasing
from top to bottom.

\subsubsection{(determinantal) Young bitableaux}

Let $S = (\omega_1, \omega_2, \ldots, \omega_p)$ and $T = (\varpi_1, \varpi_2, \ldots, \varpi_p)$ be Young tableaux on
$L = \{x_1, x_2, \ldots, x_n \}$ and $P = \{1, 2, \ldots, d \}$ of shapes $\lambda$ and $\mu$, respectively.

Following again  \cite{rota-BR} and \cite{Brini1-BR},  the (determinantal) {\it{Young bitableau}}
\begin{equation}\label{bitableaux}
(S|T) =
\left(
\begin{array}{c}
\omega_1\\ \omega_2\\ \vdots\\ \omega_p
\end{array}
\right| \left.
\begin{array}{c}
\varpi_1\\ \varpi_2\\ \vdots\\  \varpi_p
\end{array}
\right)
\end{equation}
is the element of ${\mathbb C}[M_{n,d}]$ defined  in the following way:

\begin{itemize}

\item [--]
If $\lambda = \mu$, the (determinantal) {\it{Young bitableau}} $(S|T)$
is the \emph{signed} product of the biproducts of pairs of corresponding  rows:
\begin{equation}\label{bitableau}
(S|T) =
\pm \ (\omega_1|\varpi_1)(\omega_2|\varpi_2) \cdots (\omega_p|\varpi_p),
\end{equation}
where
\begin{equation}\label{crossing sign}
\pm  = (-1)^{\ell(\omega_2)\ell(\varpi_1)+\ell(\omega_3)(\ell(\varpi_1)+\ell(\varpi_2))
+ \cdots +\ell(\omega_p)(\ell(\varpi_1)+\ell(\varpi_2)+\cdots+\ell(\varpi_{p-1}))},
\end{equation}
and the symbol $\ell(w)$ denotes the length of the word $w$.

\item [--] If $\lambda \neq \mu$, the {\it{Young bitableau}} $(S|T)$ is set to be zero.
\end{itemize}

\subsubsection{Column bitableaux in ${\mathbb C}[M_{n,d}]$}

A {\it{column}} tableau is a Young tableau of shape $\lambda = (1, 1, \ldots, 1) \vdash h$, and the number $h$ of
$1'$s is called the {\it{depth}}.

A column bitableau in ${\mathbb C}[M_{n,d}]$ is a (determinantal) bitableau $(S|T)$, where $S$ and $T$ are column Young
tableaux of the same depth. A column bitableau of depth $h$ equals, up to a sign, a {\it{monomial}} in
${\mathbb C}[M_{n,d}]$:
\begin{equation}\label{column sign}
\left(
\begin{array}{c}
i_1\\  i_2 \\ \vdots \\ i_h
\end{array}
\right| \left.
\begin{array}{c}
j_1\\  j_2 \\ \vdots \\ j_h
\end{array}
\right)
=
(-1)^{h \choose 2} (i_1|j_1)(i_2|j_2) \cdots (i_h|j_h).
\end{equation}

Although the notion of column bitableaux may appear fairly obvious, it will play
a crucial role in the passage from the polynomial algebra ${\mathbb C}[M_{n,d}]$
to the enveloping algebra $\mathbf{U}(gl(n))$ via the \emph{bitableaux correspondence} isomorphism,
Section \ref{sect Kosz} below.

\subsubsection{Bitableaux expansion into column  bitableaux}\label{Bit Exp}

Recall that
$$
(i_1 i_2 \cdots i_h|j_1 j_2 \cdots j_h)
=
(-1)^{h \choose 2}  \ det[ (i_s|j_t) ]_{s,t =1, 2, \ldots, h}
\in {\mathbb C}[M_{n,d}],
$$
and, therefore, the biproduct $(i_1 i_2 \cdots i_h|j_1 j_2 \cdots j_h) \in {\mathbb C}[M_{n,d}]$ expands
into column bitableaux as follows:
$$
(i_1 i_2 \cdots i_h|j_1 j_2 \cdots j_h) =
\sum_{\sigma \in \mathbf{S}_h} \ (-1)^{|\sigma|} \left(
\begin{array}{c}
i_{\sigma(1)}\\  i_{\sigma(2)} \\ \vdots \\ i_{\sigma(h)}
\end{array}
\right| \left.
\begin{array}{c}
j_1\\  j_2 \\ \vdots \\ j_h
\end{array}
\right)
=\sum_{\sigma \in \mathbf{S}_h} \ (-1)^{|\sigma|}  \left(
\begin{array}{c}
i_1\\  i_2 \\ \vdots \\ i_h
\end{array}
\right| \left.
\begin{array}{c}
j_{\sigma(1)}\\  j_{\sigma(2)} \\ \vdots \\ j_{\sigma(h)}
\end{array}
\right).
$$
Notice that, in the passage from monomials to column bitableaux,
the sign $(-1)^{h \choose 2}$ disappears,  due to Eq. (\ref{column sign}).

The preceding arguments extend to  bitableaux of any shape $\lambda, \ \lambda_1 \leq n.$
Given a bitableau $(S|T) \in {\mathbb C}[M_{n,d}]$ of  shape $\lambda = (\lambda_1 \geq \lambda_2 \geq \cdots \geq \lambda_m) \vdash h$
with
$$
S = \left(
\begin{array}{llllllllllllll}
i_{p_1}  \ldots    \ldots     \ldots     i_{p_{\lambda_1}} \\
i_{q_1}   \ldots  \ldots               i_{q_{\lambda_2}} \\
 \ldots  \ldots   \\
i_{r_1} \ldots i_{r_{\lambda_m}}
\end{array}
\right),
\quad
T = \left(
\begin{array}{llllllllllllll}
j_{s_1}  \ldots    \ldots     \ldots     j_{s_{\lambda_1}} \\
j_{t_1}   \ldots  \ldots               j_{t_{\lambda_2}} \\
 \ldots  \ldots   \\
j_{v_1} \ldots j_{v_{\lambda_m}}
\end{array}
\right),
$$
we have
$$
(S|T) =
\sum_{\sigma_1, \ldots, \sigma_m } \ (-1)^{\sum_{k=1}^m \ |\sigma_k|} \
\left(
\begin{array}{c}
i_{p_{\sigma_1(1)}}\\   \vdots \\ i_{p_{\sigma_1(\lambda_1)}} \\
\vdots \\ \vdots  \\
i_{r_{\sigma_m(1)}}\\   \vdots \\ i_{r_{\sigma_m(\lambda_m)}}
\end{array}
\right| \left.
\begin{array}{c}
j_{s_1}\\   \vdots \\ j_{s_{\lambda_1}}  \\
\vdots \\ \vdots \\
j_{v_1}\\   \vdots \\ j_{v_{\lambda_m}}
\end{array}
\right)
$$
$$
\phantom{(S|T)} =
\sum_{\sigma_1, \ldots, \sigma_m } \ (-1)^{\sum_{k=1}^m \ |\sigma_k|} \
\left(
\begin{array}{c}
i_{p_1}\\   \vdots \\ i_{p_{\lambda_1}}  \\
\vdots \\ \vdots \\
i_{r_1}\\   \vdots \\ i_{r_{\lambda_m}}
\end{array}
\right| \left.
\begin{array}{c}
j_{s_{\sigma_1(1)}}\\   \vdots \\ j_{s_{\sigma_1(\lambda_1)}} \\
\vdots \\ \vdots  \\
j_{v_{\sigma_m(1)}}\\   \vdots \\ j_{v_{\sigma_m(\lambda_m)}}
\end{array}
\right)
$$
where the multiple sums range over all permutations
$\sigma_1 \in \mathbf{S}_{\lambda_1}, \ldots, \sigma_m \in \mathbf{S}_{\lambda_m}.$

Notice that only the signs of permutations remain.

\subsubsection{The straightening algorithm and the standard basis of ${\mathbb C}[M_{n,d}]$}\label{straight},

Given a positive integer $h \in \mathbb{Z}^+$,
let ${\mathbb C}_h[M_{n,d}]$ denote the $h-$th homogeneous component of
${\mathbb C}[M_{n,d}].$

Consider the set of all bitableaux $(S|T) \in {\mathbb C}_h[M_{n,d}]$, where $sh(S) = sh(T) \vdash h$.
In the following, let
denote by $\leq$ the linear order on this set defined  by the following two steps:
\begin{itemize}

\item[--]
$(S|T) < (S'|T')$ whenever $sh(S)  <_l sh(S') $, where $<_l$ denotes the lexicographic
order on partitions $\lambda \vdash h.$

\item[--]
$(S|T) < (S'|T')$ whenever $sh(S) =  sh(S')$, $w(S)w(T) >_l w(S')w(T')$.
\end{itemize}
where the shapes and the concatenated words $w(S)w(T),  w(S')w(T')$ of the tableaux $S, T$  and $S', T'$
(see Eq. (\ref{word}))  are compared in the lexicographic order.

The next Theorem is a  well-known result for the polynomial algebra ${\mathbb C}[M_{n,d}]$
(\cite{drs-BR}, \cite{DKR-BR}, \cite{DEP-BR},
for the general theory of standard monomials see, e.g. \cite{Procesi-BR}, Chapt. 13).

\begin{theorem} {\bf{(The Standard basis theorem for ${\mathbb C}_h[M_{n,d}]$)}} \label{theorem: standard basis}
\begin{itemize}
\item[--]
The  set
$$
\{ (S|T) \ standard;\ sh(S) = sh(T) = \lambda \vdash h,    \lambda_1 \leq n, d \ \}.
$$
is a basis of ${\mathbb C}_h[M_{n,d}]$.
\item[--]
Furthermore, a  Young bitableau $(P|Q) \in {\mathbb C}_h[M_{n,d}]$ can be uniquely
written as a linear combination
\begin{equation}\label{straightening}
(P|Q) = \sum_{S,T} a_{S,T}\ (S|T),
\end{equation}
of standard bitableaux $(S|T)$, where
\begin{itemize}
\item[--]
the coefficient $a_{S,T} = 0$ whenever
$(S|T) \ngeq (P|Q)$;
\item[--]
the contents of the tableaux are preserved, that is
$c_S = c_P$, $c_T = c_Q$.
\end{itemize}
\end{itemize}

\end{theorem}

For a proof, see e.g. \cite{DKR-BR}, \cite{DEP-BR}.

\subsection{Right symmetrized  bitableaux and the Gordan-Capelli basis of ${\mathbb C}[M_{n,d}]$}\label{symm bit}

Given a Young tableau $T$, we say that another tableau $\overline{T}$ is a
{\textit{column permuted}} of $T$ whenever each column of $\overline{T}$  can
obtained by permuting the corresponding column of $T$.

A {\textit{right symmetrized bitableau}} $(S|\fbox{$T$})$ is the element of the
polynomial algebra ${\mathbb C}[M_{n,d}]$ defined as the following sum of bitableaux:
$$
(S|\fbox{$T$}) = \sum_{\overline{T}} \ (S|\overline{T}),
$$
where the sum is extended over {\textit{all}} $\overline{T}$ column permuted of $T$ (hence, repeated entries in
a column give rise to multiplicities).

\begin{example}\label{ex symm}
\begin{align*}
\left(
\begin{array}{cc}
 1 & 3 \\  2 & 4
\end{array}
\right| \left.
\fbox{$
\begin{array}{cc}
1 & 2 \\ 1 & 3
\end{array}
$} \
\right)
& =
\left(
\begin{array}{cc}
 1 & 3 \\  2 & 4
\end{array}
\right| \left.
\begin{array}{cc}
1 & 2 \\ 1 & 3
\end{array}
\right)
+
\left(
\begin{array}{cc}
 1 & 3 \\  2 & 4
\end{array}
\right| \left.
\begin{array}{cc}
1 & 2 \\ 1 & 3
\end{array}
\right)
\\
&  +
\left(
\begin{array}{cc}
 1 & 3 \\  2 & 4
\end{array}
\right| \left.
\begin{array}{cc}
1 & 3 \\ 1 & 2
\end{array}
\right)
+
\left(
\begin{array}{cc}
 1 & 3 \\  2 & 4
\end{array}
\right| \left.
\begin{array}{cc}
1 & 3 \\ 1 & 2
\end{array}
\right)
\\
& =
2  \left(
\begin{array}{cc}
 1 & 3 \\  2 & 4
\end{array}
\right| \left.
\begin{array}{cc}
1 & 2 \\ 1 & 3
\end{array}
\right)
+
2  \left(
\begin{array}{cc}
 1 & 3 \\  2 & 4
\end{array}
\right| \left.
\begin{array}{cc}
1 & 3 \\ 1 & 2
\end{array}
\right).
\end{align*}
\end{example}\qed

We recall a fundamental result:

\begin{theorem}{\bf{(The Gordan-Capelli basis of ${\mathbb C}[M_{n,d}]$})}\label{GC basis}
Let $h \in {\mathbb N}.$
\begin{itemize}
\item [--]
The set
$$
\{ (S|\fbox{$T$}); S, T \ standard, \ sh(S) = sh(T) = \lambda \vdash h, \ \lambda_1 \leq n, d \}
$$
is a basis of ${\mathbb C}_h[M_{n,d}]$.
\item [--]
Any right symmetrized bitableau $(U|\fbox{$V$})$, $sh(U) = sh(V) = \lambda \vdash h$,
(uniquely) expands into a linear combination of right symmetrized bitableau $(S|\fbox{$T$})$, $S, T$ standard
of the \textbf{same} shape $\lambda = sh(S) = sh(T)$.
\item [--]
Let  $(U|\fbox{$V$})$, $sh(U) = sh(V) = \lambda \vdash h$, $\lambda_1 \nleq n, d$. Then
$$
(U|\fbox{$V$}) = 0.
$$
\end{itemize}
\end{theorem}

\begin{corollary}\label{GC decomposition}
The subspace ${\mathbb C}_h[M_{n,d}]$ decomposes as:
\begin{equation}\label{eq decomposition}
{\mathbb C}_h[M_{n,d}] = \bigoplus_{\lambda \vdash h} \ {\mathbb C}_h^{\lambda}[M_{n,d}], \quad \lambda_1 \leq n, d,
\end{equation}
where ${\mathbb C}_h^{\lambda}[M_{n,d}]$ is the subspace spanned by the right symmetrized bitableaux
$(U|\fbox{$V$})$ of shape $\lambda = sh(U) = sh(V)$.
\end{corollary}

Theorem \ref{GC basis}  was proved, in a different language, by Wallace \cite{Wall-BR} in the  classical commutative case.
A superalgebraic version of this result was proved by the present authors in \cite{Brini1-BR};
for a more detailed discussion, see  \cite{Bri-BR}.

\subsection{Right symmetrized bitableaux in ${\mathbb C}_h[M_{n,d}]$, Young symmetrizers
and the natural units in the group algebra $\mathbb{C}[\mathbf{S}_h]$}

In this subsection, we summarize some basic notions from  the representation theory of the symmetric group;
furthermore, we provide useful descriptions of
{\textit{right symmetrized determinantal bitableaux}} in terms of \emph{Young symmetrizers}
and of the \emph{natural units} in the group algebra $\mathbb{C}[\mathbf{S}_h]$.

\subsubsection{The symmetric group $\mathbf{S}_h$}\label{symmgroup}

Our main reference here is the treatise of James and Kerber \cite{JK}, Chapter $3$, with the proviso
that here the \emph{role of rows and columns of a Young tableau are interchanged}.

Given a pair
$\mathbf{S}, \mathbf{T}$ of multilinear  tableaux of the same
shape $sh(\mathbf{S}) = sh(\mathbf{T}) = \lambda \vdash h$, the {\textit{Young symmetrizer}}
$\mathbf{e}^{\lambda}_{\mathbf{S}\mathbf{T}} \in \mathbb{C}[\mathbf{S}_h]$ is the element:
\begin{equation}\label{Young symmetrizers}
\mathbf{e}^{\lambda}_{\mathbf{S}\mathbf{T}} = \sum_{\sigma \in R(\mathbf{S}), \tau \in C(\mathbf{T})} \
(-1)^{|\sigma|} \ \sigma \ \theta_{\mathbf{S}\mathbf{T}} \ \tau,
\end{equation}
where $\theta_{\mathbf{S}\mathbf{T}}$ is the permutation of $\underline{h}$ such that
$\theta_{\mathbf{S}\mathbf{T}}(i) = (\mathbf{S} \circ \mathbf{T}^{-1})(i) =
\mathbf{S} \big(\mathbf{T}^{-1}(i) \big)$ for every $i = 1, 2, \ldots, h,$ and
$R(\mathbf{S}), \ C(\mathbf{T}) \subseteq \mathbf{S}_h$ are the \emph{row subgroup} of $\mathbf{S}$ and
the \emph{column subgroup} of $\mathbf{T}$, respectively.

Clearly,
$$
\mathbf{e}^{\lambda}_{\mathbf{S}\mathbf{T}} = \theta_{\mathbf{S}\mathbf{T}} \ \mathbf{e}^{\lambda}_{\mathbf{T}\mathbf{T}} =
\mathbf{e}^{\lambda}_{\mathbf{S}\mathbf{S}} \ \theta_{\mathbf{S}\mathbf{T}}.
$$

\begin{remark}\label{conj remark}
By \emph{Eq. (\ref{Young symmetrizers})}, the \emph{trivial representation} is \emph{here} associated to the \emph{column} shape
$\lambda = (1^h)$ and the \emph{sign representation} is \emph{here} associated to the \emph{row} shape
$\lambda = (h).$ \qed
\end{remark}

 We will denote by $\mathbf{\gamma}^{\lambda}_{\mathbf{S}\mathbf{T}}$ the  {\textit{natural units}}
of the group algebra $\mathbb{C}[\mathbf{S}_h]$, $\lambda \vdash h$, $\mathbf{S},  \mathbf{T}$ multilinear
standard tableaux of the same shape $sh(\mathbf{S}) = sh(\mathbf{T}) = \lambda \vdash h.$

Given $\lambda \vdash h,$ recall that
$$
\mathbb{C}[\mathbf{S}_h] = \bigoplus_{\lambda \vdash h} \ \mathbb{C}^{\lambda}[\mathbf{S}_h],
$$
where $\mathbb{C}^{\lambda}[\mathbf{S}_h]$ denotes the {\textit{isotypic}} (simple) component
of $\mathbb{C}[\mathbf{S}_h]$ associated to $\lambda$.

\begin{proposition}\label{symm repr}

\

\begin{enumerate}

\item [1)] The set
$$
\Big\{ \mathbf{e}^{\lambda}_{\mathbf{S}\mathbf{T}}; \mathbf{S}, \mathbf{T} \ multilinear \ standard \ tableaux,
sh(\mathbf{S}) = sh(\mathbf{T}) = \lambda \vdash h  \Big\}
$$
is a basis of $\mathbb{C}^{\lambda}[\mathbf{S}_h]$.

\item [2)] The set
$$
\Big\{ \mathbf{\gamma}^{\lambda}_{\mathbf{S}\mathbf{T}}; \mathbf{S}, \mathbf{T} \ multilinear \ standard \ tableaux,
sh(\mathbf{S}) = sh(\mathbf{T}) = \lambda \vdash h  \Big\}
$$
is a basis of $\mathbb{C}^{\lambda}[\mathbf{S}_h]$.

\item [3)]
Let $\mathbf{S}, \mathbf{S'}, \mathbf{T}, \mathbf{T'}$ be multilinear standard tableaux of shape $\lambda \vdash h$, then
$$
\mathbf{\gamma}^{\lambda}_{\mathbf{S}\mathbf{T}} \ \mathbf{\gamma}^{\lambda}_{\mathbf{S'}\mathbf{T'}} =
\delta_{\mathbf{T}, \mathbf{S'}}  \    \mathbf{\gamma}^{\lambda}_{\mathbf{S}\mathbf{T'}},
$$
$$
\mathbf{\gamma}^{\lambda}_{\mathbf{S}\mathbf{T}} \ \mathbf{e}^{\lambda}_{\mathbf{S'}\mathbf{T'}} =
\delta_{\mathbf{T}, \mathbf{S'}}  \    \mathbf{e}^{\lambda}_{\mathbf{S}\mathbf{T'}}.
$$
\end{enumerate}
\end{proposition}

Let $\lambda \vdash h$ be a partition and denote by $\chi^{\lambda}$ the {\textit{irreducible character}} associated
to the irreducible representation of shape $\lambda$ of the symmetric group $\mathbf{S}_h$.
Let
\begin{equation}\label{character}
\overline{\chi}_{\lambda} = \sum_{\sigma \in \mathbf{S}_h} \ \chi^{\lambda}(\sigma) \sigma \in {\mathbb C}[\mathbf{S}_h].
\end{equation}

\begin{proposition}

\

\begin{enumerate}

\item [1)]
The elements
$$
\frac {\chi^{\lambda}(I)} {n!} \ \overline{\chi}_{\lambda}, \quad \lambda \vdash h
$$
are the \emph{primitive central idempotents} of  ${\mathbb C}[\mathbf{S}_h]$.

\item [2)]
We have:
$$\label{basic identity}
\frac {\chi^{\lambda}(I)} {n!} \ \overline{\chi}_{\lambda} =
\sum_{\mathbf{T}} \ \mathbf{\gamma}^{\lambda}_{\mathbf{T} \mathbf{T}},
$$
where the sum  ranges over all multilinear $\mathbf{T}$
standard tableaux on $\underline{h}$ of shape $\lambda$,
and
$$
\frac {\chi^{\lambda}(I)} {n!} = \frac {1} {H(\lambda)},
$$
$H(\lambda)$
the {\textit{hook number}} of the partition $\lambda.$

\item [3)]
The elements
$$
\frac {\chi^{\lambda}(I)} {n!} \ \overline{\chi}_{\lambda}, \quad \lambda \vdash h
$$
are the \emph{projectors} from ${\mathbb C}[\mathbf{S}_h]$ to the isotypic (simple)
components $\mathbb{C}^{\lambda}[\mathbf{S}_h]$.
\end{enumerate}
\end{proposition}
\begin{proof}
Assertion $1)$ is an instance of a standard fact of the representation theory of finite groups.
Assertions $2)$ and $3)$ follow from assertion $1)$ and Proposition
\ref{symm repr}.
\end{proof}

\subsubsection{Right symmetrized bitableaux and Young symmetrizers: the multilinear case}

We consider the algebra ${\mathbb C}_h[M_{h,h}]$, that is $n= d =h$, the polynomial algebra generated
by the variables $(i|j)$, $i, j =1, 2, \ldots, h$.

We establish the following convention. Given an element
$$
\mathbf{p} = \sum _s \ c_s \sigma_s \in \mathbb{C}[\mathbf{S}_h],
$$
and a column tableau
$$
\left(
\begin{array}{c}
I (1) \\\   \vdots \\ I (h)
\end{array}
\right| \left.
\begin{array}{c}
 J(1)\\   \vdots \\  J(h)
\end{array}
\right),
$$
we set
\begin{equation}\label{sum convention}
\left(
\begin{array}{c}
I \big( \mathbf{p} (1) \big) \\   \vdots \\ I \big( \mathbf{p} (h) \big)
\end{array}
\right| \left.
\begin{array}{c}
 J(1)\\   \vdots \\  J(h)
\end{array}
\right)
=
\sum_s \ c_s
\left(
\begin{array}{c}
I \big( \sigma_s (1) \big)\\   \vdots
\\ I \big( \sigma_s (h) \big)
\end{array}
\right| \left.
\begin{array}{c}
 J(1)\\   \vdots \\  J(h)
\end{array}
\right).
\end{equation}

\

\begin{proposition}\label{multilinear}
Let $\mathbf{S}, \mathbf{T}$ be multilinear tableaux of the same shape $\lambda$,
then
\begin{equation}
(\mathbf{S}|\fbox{$\mathbf{T}$}) =
\left(
\begin{array}{c}
\mathbf{e}^{\lambda}_{\mathbf{S}\mathbf{T}}(1) \\   \vdots \\ \mathbf{e}^{\lambda}_{\mathbf{S}\mathbf{T}}(h)
\end{array}
\right| \left.
\begin{array}{c}
1 \\  \vdots \\ h
\end{array}
\right).
\end{equation}
\end{proposition}

\begin{proof}

Since
$$
(\mathbf{S}|\mathbf{T}) = \sum_{\sigma \in R(\mathbf{S})} \
(-1)^{|\sigma|}
\left(
\begin{array}{c}
\sigma \ \mathbf{S}(1)\\   \vdots \\ \sigma \ \mathbf{S}(h)
\end{array}
\right| \left.
\begin{array}{c}
\mathbf{T}(1)\\   \vdots \\ \mathbf{T}(h)
\end{array}
\right)
$$
(see subsection \ref{Bit Exp}),
then
\begin{align*}
(\mathbf{S}|\fbox{$\mathbf{T}$}) &= \sum_{\sigma \in R(\mathbf{S}), \tau \in C(\mathbf{T})} \
(-1)^{|\sigma|}
\left(
\begin{array}{c}
\sigma \ \mathbf{S}(1)\\   \vdots \\ \sigma \ \mathbf{S}(h)
\end{array}
\right| \left.
\begin{array}{c}
\tau \ \mathbf{T}(1)\\   \vdots \\ \tau \ \mathbf{T}(h)
\end{array}
\right)
\\
&= \sum_{\sigma \in R(\mathbf{S}), \tau \in C(\mathbf{T})} \
(-1)^{|\sigma|}
\left(
\begin{array}{c}
 \sigma \ \theta_{\mathbf{S}\mathbf{T}} \ \mathbf{T}(1)\\
\vdots \\   \sigma \ \theta_{\mathbf{S}\mathbf{T}} \ \mathbf{T}(h)
\end{array}
\right| \left.
\begin{array}{c}
\tau \ \mathbf{T}(1)\\   \vdots \\ \tau \ \mathbf{T}(h)
\end{array}
\right)
\\
&=
\sum_{\sigma \in R(\mathbf{S}), \tau \in C(\mathbf{T})} \
(-1)^{|\sigma|}
\left(
\begin{array}{c}
\sigma \  \theta_{\mathbf{S}\mathbf{T}} \ \tau^{-1} \ \mathbf{T}(1)\\   \vdots \\ \sigma \
\theta_{\mathbf{S}\mathbf{T}} \ \tau^{-1} \ \mathbf{T}(h)
\end{array}
\right| \left.
\begin{array}{c}
 \mathbf{T}(1)\\   \vdots \\  \mathbf{T}(h)
\end{array}
\right)
\\
&=
\sum_{\sigma \in R(\mathbf{S}), \tau \in C(\mathbf{T})} \
(-1)^{|\sigma|}
\left(
\begin{array}{c}
\sigma \  \theta_{\mathbf{S}\mathbf{T}} \ \tau (1) \\   \vdots \\ \sigma \  \theta_{\mathbf{S}\mathbf{T}} \ \tau  (h)
\end{array}
\right| \left.
\begin{array}{c}
 1\\   \vdots \\  h
\end{array}
\right).
\end{align*}
Since
$$
\mathbf{e}^{\lambda}_{\mathbf{S}\mathbf{T}} = \sum_{\sigma \in R(\mathbf{S}), \tau \in C(\mathbf{T})} \
(-1)^{|\sigma|} \ \sigma \ \theta_{\mathbf{S}\mathbf{T}} \ \tau,
$$
then
$$
(\mathbf{S}|\fbox{$\mathbf{T}$})
=
\left(
\begin{array}{c}
\mathbf{e}^{\lambda}_{\mathbf{S}\mathbf{T}}(1) \\   \vdots \\ \mathbf{e}^{\lambda}_{\mathbf{S}\mathbf{T}}(h)
\end{array}
\right| \left.
\begin{array}{c}
1 \\  \vdots \\ h
\end{array}
\right).
$$
\end{proof}

\subsubsection{Right symmetrized bitableaux and Young symmetrizers: the general case}

Let $U, V$ be Young tableaux on the alphabets $\underline{n}, \underline{d}$, and let
$\mathbf{S}, \mathbf{T}$ be multilinear tableaux of the same shape $\lambda \vdash h$.
There exists a unique  pair of maps $I :  \underline{h}   \rightarrow \underline{n}$, $J :  \underline{h}   \rightarrow \underline{d}$,
such that
$$
U = I_{\mathbf{S}} \quad V =J_{\mathbf{T}}
$$
(see Eq. (\ref{specialization})).

\begin{proposition}
\begin{equation}\label{gen case}
(U|\fbox{$V$}) =
(I_{\mathbf{S}}|\fbox{$ J_{\mathbf{T}}$}) =
\left(
\begin{array}{c}
I \big( \mathbf{e}^{\lambda}_{\mathbf{S}\mathbf{T}} (1) \big) \\   \vdots \\ I \big( \mathbf{e}^{\lambda}_{\mathbf{S}\mathbf{T}} (h) \big)
\end{array}
\right| \left.
\begin{array}{c}
 J(1)\\   \vdots \\  J(h)
\end{array}
\right).
\end{equation}
\end{proposition}

\begin{proof}

From  Proposition \ref{multilinear}, we get:

\begin{align*}
(I_{\mathbf{S}}|\fbox{$ J_{\mathbf{T}}$}) &= \sum_{\eta \in R(\mathbf{T}), \tau \in C(\mathbf{T})} \
(-1)^{|\eta|}
\left(
\begin{array}{c}
I \big(   \eta \ \theta_{\mathbf{S}\mathbf{T}} \ \tau^{-1} \ (1) \big)\\   \vdots
\\ I \big( \eta \ \theta_{\mathbf{S}\mathbf{T}} \ \tau^{-1} (h) \big)
\end{array}
\right| \left.
\begin{array}{c}
 J(1)\\   \vdots \\  J(h)
\end{array}
\right).
\end{align*}
\end{proof}

\subsection{Immanants in ${\mathbb C}_h[M_{n,d}]$}

The immanant of a matrix was defined by D. E. Littlewood and A. R. Richardson as a
generalization of the concepts of determinant and permanent  \cite{Little1-BR} (see also
\cite{Little2-BR}, \cite{goulden-BR}).

Let $\lambda \vdash h$ be a partition and denote by $\chi^{\lambda}$ the {\textit{irreducible character}} associated
to the irreducible representation of shape $\lambda$ of the symmetric group $\mathbf{S}_h$,
and let
$$
\overline{\chi}_{\lambda} = \sum_{\sigma \in \mathbf{S}_h} \ \chi^{\lambda}(\sigma) \sigma \in {\mathbb C}[\mathbf{S}_h].
$$

\begin{example}
Let $n = 3$, $\lambda = (2, 1) \vdash 3$.

The irreducible character element of the group algebra ${\mathbb C}[\mathbf{S}_3]$ associated to the partition $\lambda = (2, 1)$ is
$$
\overline{\chi}_{\lambda} = \sum_{\sigma \in \mathbf{S}_3} \ \chi^{\lambda}(\sigma) \sigma = 2 I - (123) - (132) \in \mathbb{C}[\mathbf{S}_3].
$$
\end{example}\qed

The  {\textit{(generalized) immanant}}
$$
imm_{\lambda}(i_1 i_2 \cdots i_h;j_1 j_2 \cdots j_h), \quad (i_1 i_2 \cdots i_h) \in \underline{n}^h, \ (j_1 j_2 \cdots j_h) \in \underline{d}^h
$$
in  the polynomial algebra in ${\mathbb C}_h[M_{n,d}]$ is the element:

\begin{align*}
imm_{\lambda}(i_1 i_2 \cdots i_h;j_1 j_2 \cdots j_h) & =
\sum_{\sigma \in \mathbf{S}_h} \ \chi^{\lambda}(\sigma) \left(
\begin{array}{c}
i_{\sigma(1)}\\  i_{\sigma(2)} \\ \vdots \\ i_{\sigma(h)}
\end{array}
\right| \left.
\begin{array}{c}
j_1\\  j_2 \\ \vdots \\ j_h
\end{array}
\right)
\\
& =\sum_{\sigma \in \mathbf{S}_h} \ \chi^{\lambda}(\sigma)  \left(
\begin{array}{c}
i_1\\  i_2 \\ \vdots \\ i_h
\end{array}
\right| \left.
\begin{array}{c}
j_{\sigma(1)}\\  j_{\sigma(2)} \\ \vdots \\ j_{\sigma(h)}
\end{array}
\right).
\end{align*}

Since the characters are {\textit{invariant on the  conjugacy classes }} of $\mathbf{S}_h$, it follows that
$$
imm_{\lambda}(i_{\tau(1)}  i_{\tau(2)}  \cdots  i_{\tau(h)};j_{\tau(1)}  j_{\tau(2)}  \cdots  j_{\tau(h)}) =
 imm_{\lambda}(i_1 i_2 \cdots i_h;j_1 j_2 \cdots j_h).
$$

Hence,

\begin{proposition}\label{IMM}
The map
$$
IMM_{\lambda} :\ \left(
\begin{array}{c}
i_1\\ i_2 \\ \vdots \\ i_h
\end{array}
\right| \left.
\begin{array}{c}
j_1\\  j_2 \\ \vdots \\ j_h
\end{array}
\right)
\mapsto
imm_{\lambda}(i_1 i_2 \cdots i_h;j_1 j_2 \cdots j_h)
$$
defines a linear map
$$
IMM_{\lambda} : {\mathbb C}_h[M_{n,d}] \rightarrow   {\mathbb C}_h[M_{n,d}].
$$
\end{proposition}

Clearly, the immanant $imm_{\lambda}(i_1 i_2 \cdots i_h;j_1 j_2 \cdots j_h) \in {\mathbb C}_h[M_{n,d}]$ is the
{\textit{natural}} generalization of the  biproducts (\emph{signed minors}) $(i_1 i_2 \cdots i_h|j_1 j_2 \cdots j_h)$
in ${\mathbb C}[M_{n,d}].$

It is obvious that the immanants $imm_{\lambda}(i_1 i_2 \cdots i_h;j_1 j_2 \cdots j_h)$ are
homogeneous elements of degree $h \in \mathbb{N}$ of the polynomial algebra ${\mathbb C}[M_{n,d}]$;
therefore, by Theorem \ref{GC basis}, the immanants $imm_{\lambda}(i_1 i_2 \cdots i_h;j_1 j_2 \cdots j_h)$
expand into linear combination of standard right symmetrized bitableaux of
shapes that are partitions of $h$.

Furthermore, the following  stronger result holds.

\begin{proposition}\label{imm vs bit}
Let $\lambda \vdash h$. Any immanant $imm_{\lambda}(i_1 i_2 \cdots i_h;j_1 j_2 \cdots j_h)$
can be written as  a linear combination of standard right symmetrized bitableaux of the \textbf{same} shape $\lambda$:
\begin{multline*}
imm_{\lambda}(i_1 i_2 \cdots i_h;j_1 j_2 \cdots j_h) =
 \sum_{U, V} \ \varrho_{U,V}  \ (U|\fbox{$V$}), \\ \varrho_{U,V} \in {\mathbb C}, \quad sh(U) = sh(V) = \lambda.
\end{multline*}
\end{proposition}

\begin{proof}

Let $I: \underline{h} \rightarrow \underline{n}, \   J:  \underline{h}  \rightarrow \underline{n}$,
$I(s) = i_s, \ J(s) = j_s, \quad s = 1, 2, \ldots, h$.

Item $4)$ of Proposition \ref{symm repr} implies:
\begin{align*}
imm_{\lambda}(i_1 i_2 \cdots i_h;j_1 j_2 \cdots j_h)
&=&
\sum_{\sigma \in \mathbf{S}_h} \ \chi^{\lambda}(\sigma) \left(
\begin{array}{c}
I(\sigma(1))\\ I(\sigma(2)) \\ \vdots \\ I(\sigma(h))
\end{array}
\right| \left.
\begin{array}{c}
J(1)\\  J(2) \\ \vdots \\ J(h)
\end{array}
\right)
\\
&=&
\left(
\begin{array}{c}
I(\chi_{\lambda}(1))\\  I(\chi_{\lambda}(2)) \\ \vdots \\ I(\chi_{\lambda}(h))
\end{array}
\right| \left.
\begin{array}{c}
J(1)\\  J(2) \\ \vdots \\ J(h)
\end{array}
\right)
\\
&=&
H(\lambda) \
\sum_{\mathbf{T}} \  \left(
\begin{array}{c}
I({\mathbf{\gamma}}^{\lambda}_{\mathbf{T}\mathbf{T}}(1))\\ I({\mathbf{\gamma}}^{\lambda}_{\mathbf{T}\mathbf{T}}(2))
\\ \vdots \\ I({\mathbf{\gamma}}^{\lambda}_{\mathbf{T}\mathbf{T}}(h))
\end{array}
\right| \left.
\begin{array}{c}
J(1)\\  J(2) \\ \vdots \\ J(h)
\end{array}
\right)
\end{align*}
Since the natural units ${\mathbf{\gamma}}^{\lambda}_{\mathbf{T} \mathbf{T}}$ expand into Young symmetrizers
of the same shape:
$$
{\mathbf{\gamma}}^{\lambda}_{\mathbf{T} \mathbf{T}} = \sum_{\mathbf{S_1}, \mathbf{S_2}} \
C_{\mathbf{T} \mathbf{T},\mathbf{S_1} \mathbf{S_2}} \cdot
\mathbf{e}^{\lambda}_{\mathbf{S_1}\mathbf{S_2}}, \quad C_{\mathbf{T} \mathbf{T},\mathbf{S_1} \mathbf{S_2}} \in \mathbb{C},
$$
then
\begin{align*}
imm_{\lambda}(i_1 i_2 \cdots i_h;j_1 j_2 \cdots j_h)
&=
H(\lambda) \
\sum_{\mathbf{T}} \
\sum_{\mathbf{S_1}, \mathbf{S_2}} \
C_{\mathbf{T} \mathbf{T},\mathbf{S_1} \mathbf{S_2}}
\left(
\begin{array}{c}
I(\mathbf{e}^{\lambda}_{\mathbf{S_1}\mathbf{S_2}}(1))\\  I(\mathbf{e}^{\lambda}_{\mathbf{S_1}\mathbf{S_2}}(2))
\\ \vdots \\ I(\mathbf{e}^{\lambda}_{\mathbf{S_1}\mathbf{S_2}}(h))
\end{array}
\right| \left.
\begin{array}{c}
J(1)\\  J(2) \\ \vdots \\ J(h)
\end{array}
\right)
\\
&=
H(\lambda) \
\sum_{\mathbf{T}} \
\sum_{\mathbf{S_1}, \mathbf{S_2}} \
C_{\mathbf{T} \mathbf{T},\mathbf{S_1} \mathbf{S_2}}
(I_{S_1}|\fbox{$J_{S_2}$}).
\end{align*}
\end{proof}

From Proposition \ref{imm vs bit} and Theorem \ref{GC basis}, it follows:

\begin{corollary} \label{zero imm}
Let $\lambda \vdash h$. If $\lambda_1  \nleq min\{n, d \}$, then
$$
imm_{\lambda}(i_1 i_2 \cdots i_h;j_1 j_2 \cdots j_h) = 0.
$$
\end{corollary} \qed

The scalar multiple  $\frac {\chi^{\lambda}(I)} {n!} \ IMM_{\lambda}$ of the
linear operator $IMM_{\lambda}$ of Proposition \ref{IMM} acts on ${\mathbb C}_h[M_{n,d}]$ as
the {\textit{projector}} on the direct summand ${\mathbb C}_h^{\lambda}[M_{n,d}]$ in the Gordan-Capelli
direct sum decomposition (\ref{eq decomposition})
of Corollary \ref{GC decomposition}.

\begin{proposition}
Let $U, V $ be  Young tableaux of the same shape $sh(U) =sh(V) = \mu \vdash h.$ We have:
\begin{enumerate}
\item
if $\mu = \lambda$, then
\begin{equation}
\frac {\chi^{\lambda}(I)} {n!} \ IMM_{\lambda} \Big( (U| \fbox{$V$}) \Big) = (U| \fbox{$V$});
\end{equation}
\item
if $\mu \neq \lambda$, then
\begin{equation}
\frac {\chi^{\lambda}(I)} {n!} \ IMM_{\lambda} \Big( (U| \fbox{$V$}) \Big) = 0.
\end{equation}
\end{enumerate}
\end{proposition}

\begin{proof}

Set
$$
U = I_{\mathbf{T_0}}, \quad V = J_{\mathbf{T_0}}, \quad sh(U) = sh(V) = sh(\mathbf{T_0}) = \mu.
$$
Equation (\ref{gen case}) implies
$$
(U| \fbox{$V$})
=
(I_{\mathbf{T_0}}|\fbox{$J_{\mathbf{T_0}}$})
=
\left(
\begin{array}{c}
I \big( \mathbf{e}^{\mu}_{\mathbf{T_0}\mathbf{T_0}} (1) \big) \\   \vdots \\ I \big( \mathbf{e}^{\mu}_{\mathbf{T_0}\mathbf{T_0}} (h) \big)
\end{array}
\right| \left.
\begin{array}{c}
 J(1)\\   \vdots \\  J(h)
\end{array}
\right).
$$

Item $4)$ of Proposition \ref{symm repr} implies

\begin{align*}
\frac {\chi^{\lambda}(I)} {n!} \ IMM_{\lambda} \Big( (U| \fbox{$V$}) \Big) &=
(\sum_{\mathbf{T}} \ \mathbf{\gamma}^{\lambda}_{\mathbf{T}\mathbf{T}})  \left(
\begin{array}{c}
I \big( \mathbf{e}^{\mu}_{\mathbf{T_0}\mathbf{T_0}} (1) \big) \\   \vdots \\ I \big( \mathbf{e}^{\mu}_{\mathbf{T_0}\mathbf{T_0}} (h) \big)
\end{array}
\right| \left.
\begin{array}{c}
 J(1)\\   \vdots \\  J(h)
\end{array}
\right)
\\
&=
\left(
\begin{array}{c}
I \big( \sum_{\mathbf{T}} \ \mathbf{\gamma}^{\lambda}_{\mathbf{T}\mathbf{T}} \ \mathbf{e}^{\mu}_{\mathbf{T_0}\mathbf{T_0}} (1) \big)
\\   \vdots \\ I \big( \sum_{\mathbf{T}} \ \mathbf{\gamma}^{\lambda}_{\mathbf{T}\mathbf{T}} \ \mathbf{e}^{\mu}_{\mathbf{T_0}\mathbf{T_0}} (h) \big)
\end{array}
\right| \left.
\begin{array}{c}
 J(1)\\   \vdots \\  J(h)
\end{array}
\right).
\end{align*}

If $\lambda \neq \mu$, the natural units $\mathbf{\gamma}^{\lambda}_{\mathbf{T}\mathbf{T}}$
and the Young symmetrizer $\mathbf{e}^{\mu}_{\mathbf{T_0}\mathbf{T_0}}$ belong
to different simple components of the semisimple algebra
$$
\mathbb{C}[\mathbf{S}_h] = \bigoplus_{\nu \vdash h} \ \mathbb{C}^{\nu}[\mathbf{S}_h],
$$
and  are therefore orthogonal. This proves the second assertion.

If $\lambda = \mu$, since (Proposition \ref{symm repr}, item $3)$ )
$$
\mathbf{\gamma}^{\lambda}_{\mathbf{T}\mathbf{T}} \ \mathbf{e}^{\lambda}_{\mathbf{T_0}\mathbf{T_0}} =
\delta_{\mathbf{T},\mathbf{T_0}} \ \mathbf{e}^{\lambda}_{\mathbf{T_0}\mathbf{T_0}},
$$
we get
$$
\frac {\chi^{\lambda}(I)} {n!} \ IMM_{\lambda} \Big( (U| \fbox{$V$}) \Big)
=
\left(
\begin{array}{c}
I \big( \mathbf{e}^{\lambda}_{\mathbf{T_0}\mathbf{T_0}} (1) \big) \\   \vdots \\ I \big( \mathbf{e}^{\lambda}_{\mathbf{T_0}\mathbf{T_0}} (h) \big)
\end{array}
\right| \left.
\begin{array}{c}
 J(1)\\   \vdots \\  J(h)
\end{array}
\right)
=
(U| \fbox{$V$}),
$$
and the first assertion is proved.
\end{proof}

\begin{proposition}\label{bit vs imm}
Let $\lambda \vdash h$. Any right symmetrized bitableau $(U| \fbox{$V$})$ of shape $sh(U) = sh(V) = \lambda$
can be written as  a linear combination of immanants $imm_{\lambda}(i_1 i_2 \cdots i_h;j_1 j_2 \cdots j_h)$
associated to the \textbf{same} shape $\lambda$.
\end{proposition}
\begin{proof}
Expand the right symmetrized bitableau $(U| \fbox{$V$})$ into monomials and
apply to each summand the linear operator $IMM_{\lambda}.$
\end{proof}

By combining Theorem \ref{GC basis} and  Proposition \ref{bit vs imm}, we get

\begin{proposition}\label{immanant span}
The set of immanants
$$
imm_{\lambda}(i_1 i_2 \cdots i_h;j_1 j_2 \cdots j_h),
$$
with
$$
\lambda \vdash h, \ \lambda_1 \nleq min\{n, d \}, \
(i_1 i_2 \cdots i_h) \in \underline{n}^h, \ (j_1 j_2 \cdots j_h) \in \underline{d}^h,
$$
is a spanning set of ${\mathbb C}_h[M_{n,d}].$
\end{proposition}

\section{The superalgebraic approach to the enveloping algebra $\mathbf{U}(gl(n))$}\label{sec 3}

In this Section, we provide a synthetic presentation   of the {\it{superalgebraic method of virtual variables}} for $gl(n)$.

This method was developed  by the present authors
for the   general linear Lie superalgebras $gl(m|n)$, in the series of notes
\cite{Bri-BR}, \cite{BriUMI-BR}, \cite{Brini1-BR}, \cite{Brini2-BR}, \cite{Brini3-BR},
\cite{Brini4-BR},  \cite{BRT-BR}.

The technique of virtual variables is  an extension of Capelli's method of  {\textit{ variabili ausilarie}}
(Capelli \cite{Cap4-BR}, see also Weyl \cite{Weyl-BR}).

Capelli introduced the technique of {\textit{ variabili ausilarie}} in order to manage symmetrizer operators in
terms of polarization operators and to  simplify the study of some skew-symmetrizer operators (namely, the famous central Capelli operator).

Capelli's idea was well suited to
treat symmetrization, but it did not work in the same efficient way while dealing with skew-symmetrization.

One had to wait the introduction of the notion of {\textit{superalgebras}} (see,e.g. \cite{Scheu-BR},  \cite{KAC1-BR})
 to have the right conceptual framework to treat symmetry and skew-symmetry in one and the same way.
To the best of our knowledge, the first mathematician who intuited the connection between Capelli's idea and superalgebras was
Koszul in $1981$  \cite{Koszul-BR};
Koszul proved that the classical determinantal Capelli operator can be rewritten - in a much simpler way -
by adding to the symbols to be dealt with an extra auxiliary symbol that obeys to different commutation relations.

The superalgebraic method of virtual variables allows us to express remarkable classes of elements in $\mathbf{U}(gl(n))$
as images - with respect to the {\textit{Capelli devirtualization epimorphism}} (Subsection \ref{Virt} below) - of simple
{\it{monomials}} and to obtain transparent combinatorial descriptions of their actions on irreducible $gl(n)-$modules.

Among these classes, here we recall the classes of  Capelli bitableaux $[S|T]$ and
Young-Capelli bitableaux $[S|\fbox{$T$}]$
(see \cite{Brini2-BR}, \cite{Brini3-BR}, \cite{Bri-BR}, and subsection \ref{Capbit sub} below),
and introduce the \emph{new} class of
Capelli immanants
$$
Cimm_{\lambda}[i_1 i_2 \cdots i_h;j_1 j_2 \cdots j_h]
$$
(see Section \ref{Capimm sect} below).

Moreover, this method
throws a  bridge between the theory of  $\mathbf{U}(gl(n))$ and  the {\it{(super)straightening techniques}}
in (super)symmetric algebras
(see, e.g. \cite{rota-BR}, \cite{Brini3-BR}, \cite{Brini4-BR},
 \cite{Bri-BR}).

\subsection{The superalgebras ${\mathbb C}[M_{m_0|m_1+n,d}]$ and $gl(m_0|m_1+n)$ }

\subsubsection{The general linear Lie super algebra $gl(m_0|m_1+n)$}
 Given a vector space $V_n$ of dimension $n$, we will regard it as a subspace of a $\mathbb{Z}_2-$graded vector space
 $W = W_0 \oplus W_1$, where
$$
W_0 = V_{m_0}, \qquad W_1 = V_{m_1} \oplus V_n.
$$
The  vector spaces
$V_{m_0}$ and $V_{m_1}$ (informally, we assume that
$dim(V_{m_0})=m_0$ and $dim(V_{m_1})=m_1$ are ``sufficiently large'') are called
the  {\textit{positive virtual (auxiliary)
vector space}},  the {\textit{negative virtual (auxiliary) vector space}}, respectively, and $V_n$ is called the {\textit{(negative) proper vector space}}.

 The inclusion $V_n \subset W$ induces a natural embedding of the ordinary general linear Lie algebra $gl(n)$ of $V_n$ into the
 {\textit{auxiliary}}
general linear Lie {\it{superalgebra}} $gl(m_0|m_1+n)$ of $W = W_0 \oplus W_1$ (see, e.g. \cite{KAC1-BR}, \cite{Scheu-BR}).

Let
$
A_0 = \{ \alpha_1, \ldots, \alpha_{m_0} \},$  $A_1 = \{ \beta_1, \ldots, \beta_{m_1} \},$
$L = \{ x_1, \ldots, x_n \}$
denote \emph{fixed  bases} of $V_{m_0}$, $V_{m_1}$ and $V_n$, respectively; therefore $|\alpha_s| = 0 \in \mathbb{Z}_2,$
and $|\beta_t| = |x_i|   = 1 \in \mathbb{Z}_2.$

Let
$$
\{ e_{a, b}; a, b \in A_0 \cup A_1 \cup L \}, \qquad |e_{a, b}| =
|a|+|b| \in \mathbb{Z}_2
$$
be the standard $\mathbb{Z}_2-$homogeneous basis of the Lie superalgebra $gl(m_0|m_1+n)$ provided by the
elementary matrices. The elements $e_{a, b} \in gl(m_0|m_1+n)$ are $\mathbb{Z}_2-$homogeneous of
$\mathbb{Z}_2-$degree $|e_{a, b}| = |a| + |b|.$

The superbracket of the Lie superalgebra $gl(m_0|m_1+n)$ has the following explicit form:
$$
\left[ e_{a, b}, e_{c, d} \right] = \delta_{bc} \ e_{a, d} - (-1)^{(|a|+|b|)(|c|+|d|)} \delta_{ad}  \ e_{c, b},
$$
$a, b, c, d \in A_0 \cup A_1 \cup L.$

\begin{remark}\label{symbols}
In the following, the elements of the sets $A_0, A_1, L$ will be called
\emph{positive virtual symbols}, \emph{negative virtual symbols} and \emph{negative proper symbols},
respectively.
\end{remark}

\subsubsection{The  supersymmetric algebra ${\mathbb C}[M_{m_0|m_1+n,d}]$}

As already said, we will write $(i|j)$ in place of $x_{ij},$ and regard the (commutative) algebra ${\mathbb C}[M_{n,d}]$
as a subalgebra of the \textit{``auxiliary'' supersymmetric algebra}
$$
{\mathbb C}[M_{m_0|m_1+n,d}] = {\mathbb C}\big[ (\alpha_s|j), (\beta_t|j), (i|j) \big]
$$
 generated by the ($\mathbb{Z}_2$-graded) variables $(\alpha_s|j), (\beta_t|j), (i|j)$,
$j = 1, 2, \ldots, d$,
 where
 $$
 |(\alpha_s|j)| = 1 \in \mathbb{Z}_2 \ \ and \ \  |(\beta_t|j)| = |(i|j)| = 0 \in \mathbb{Z}_2,
 $$
subject to the commutation relations:

$$
(a|h)(b|k) = (-1)^{|(a|h)||(b|k)|} \ (b|k)(a|h),
$$
for $a, b \in  \{ \alpha_1, \ldots, \alpha_{m_0} \} \cup \{ \beta_1, \ldots, \beta_{m_1} \} \cup \{1, 2, \ldots , n\}.$

In plain words, all the variables commute each other, with the exception of
pairs of variables $(\alpha_s|j), (\alpha_t|j)$ that skew-commute:
$$
(\alpha_s|j) (\alpha_t|j) = - (\alpha_t|j) (\alpha_s|j).
$$

In the standard notation of multilinear algebra, we have:
\begin{align*}
{\mathbb C}[M_{m_0|m_1+n,d}]
& \cong \Lambda \left[ W_0 \otimes P_d \right]
\otimes      {\mathrm{Sym}} \left[ W_1  \otimes P_d \right] \\
 & =
 \Lambda \left[ V_{m_0} \otimes P_d \right]
\otimes      {\mathrm{Sym}} \left[ (V_{m_1} \oplus V_n)  \otimes P_d \right]
\end{align*}
where $P_d = (P_d)_1$ denotes the trivially (odd) $\mathbb{Z}_2-$graded vector space with distinguished basis $\{j; \ j = 1, 2, \ldots, d \}.$

The algebra  ${\mathbb C}[M_{m_0|m_1+n,d}]$ is a supersymmetric \textit{$\mathbb{Z}_2-$graded algebra} (superalgebra), whose
$\mathbb{Z}_2-$graduation is  inherited by the natural one in the exterior algebra.

\subsubsection{Left superderivations and left superpolarizations}

A {\it{left superderivation}} $D$ ($\mathbb{Z}_2-$homogeneous of degree $|D|$) (see, e.g. \cite{Scheu-BR}, \cite{KAC1-BR}) on
${\mathbb C}[M_{m_0|m_1+n,d}]$ is an element of the superalgebra $End_\mathbb{C}[\mathbb{C}[M_{m_0|m_1+n,d}]]$
that satisfies "Leibniz rule"
$$
D(\textbf{p} \cdot \textbf{q}) = D(\textbf{p}) \cdot \textbf{q} + (-1)^{|D||\textbf{p}|} \textbf{p} \cdot D(\textbf{q}),
$$
for every $\mathbb{Z}_2-$homogeneous of degree $|\textbf{p}|$ element $\textbf{p} \in \mathbb{C}[M_{m_0|m_1+n,d}].$

Given two symbols $a, b \in A_0 \cup A_1 \cup L$, the {\textit{superpolarization}} $D_{a,b}$ of $b$ to $a$
is the unique {\it{left}} superderivation of ${\mathbb C}[M_{m_0|m_1+n,d}]$ of parity $|D_{a,b}| = |a| + |b| \in \mathbb{Z}_2$ such that
\begin{equation}
D_{a,b} \left( (c|j) \right) = \delta_{bc} \ (a|j), \ c \in A_0 \cup A_1 \cup L, \ j = 1, \ldots, d.
\end{equation}

Informally, we say that the operator $D_{a,b}$ {\it{annihilates}} the symbol $b$ and {\it{creates}} the symbol $a$.

\subsubsection{The superalgebra ${\mathbb C}[M_{m_0|m_1+n,d}]$ as a $\mathbf{U}(gl(m_0|m_1+n))$-module}

Since
$$
D_{a,b}D_{c,d} -(-1)^{(|a|+|b|)(|c|+|d|)}D_{c,d}D_{a,b} =
\delta_{b,c}D_{a,d} -(-1)^{(|a|+|b|)(|c|+|d|)}\delta_{a,d}D_{c,b},
$$
the map
$$
e_{a,b} \rightarrow D_{a,b}, \qquad a, b \in A_0 \cup A_1 \cup L.
$$
(that send the elementary
matrices to the corresponding superpolarizations) is an (even) Lie superalgebra morphism from $gl(m_0|m_1+n)$ to $End_\mathbb{C}[\mathbb{C}[M_{m_0|m_1+n,d}]]$
and, hence, it uniquely defines a
morphism (i.e. a representation):
$$
\varrho : \mathbf{U}(gl(m_0|m_1+n)) \rightarrow End_\mathbb{C}[\mathbb{C}[M_{m_0|m_1+n,d}]].
$$

In the following, we always regard the superalgebra $\mathbb{C}[M_{m_0|m_1+n,d}]$ as a $\mathbf{U}(gl(m_0|m_1+n))-$supermodule,
with respect to the action induced by the representation $\varrho$:
$$
e_{a,b} \cdot \mathbf{p} = D_{a,b}(\mathbf{p}),
$$
for every $\mathbf{p} \in {\mathbb C}[M_{m_0|m_1+n,d}].$

We recall that $\mathbb{C}[M_{m_0|m_1+n,d}]$ is  a {\it{semisimple}} $\mathbf{U}(gl(m_0|m_1+n))-$supermodule,
whose irreducible (simple) submodules are - up to isomorphism - {\it{Schur supermodules}} (see, e.g. \cite{Brini1-BR}, \cite{Brini2-BR}, \cite{Bri-BR}. For a more traditional presentation, see also \cite{CW-BR}).

Clearly, $\mathbf{U}(gl(0|n)) = \mathbf{U}(gl(n))$ is a subalgebra of $\mathbf{U}(gl(m_0|m_1+n))$
and
the subalgebra $\mathbb{C}[M_{n,d}]$ is a $\mathbf{U}(gl(n))-$submodule of  $\mathbb{C}[M_{m_0|m_1+n,d}]$.

\subsection{The virtual algebra $Virt(m_0+m_1,n)$ and the virtual
presentations of elements in $\mathbf{U}(gl(n))$}

\subsubsection{The Capelli devirtualization epimorphism $\mathfrak{p} : Virt(m_0+m_1,n) \twoheadrightarrow \mathbf{U}(gl(n))$}\label{Virt}

We say that a product
$$
e_{a_mb_m} \cdots e_{a_1b_1} \in \mathbf{U}(gl(m_0|m_1+n)), \quad a_i, b_i \in A_0 \cup A_1 \cup L, \ i= 1, \ldots, m
$$
is an {\textit{irregular expression}} whenever
  there exists a right subword
$$e_{a_i,b_i} \cdots e_{a_2,b_2} e_{a_1,b_1},$$
$i \leq m$ and a
virtual symbol $\gamma \in A_0 \cup A_1$ such that
\begin{equation}\label{irrexpr-BR}
 \# \{j;  b_j = \gamma, j \leq i \}  >  \# \{j;  a_j = \gamma, j < i \}.
\end{equation}

The meaning of an irregular expression in terms of the action of  $\mathbf{U}(gl(m_0|m_1+n))$ on
the algebra $\mathbb{C}[M_{m_0|m_1+n,d}]$ is that there exists a
virtual symbol $\gamma$ and a right subsequence in which the symbol $\gamma$ is annihilated more times than
it was already created.

\begin{example}
Let $\gamma \in  A_0 \cup A_1$ and $x_i, x_j \in L.$ The product
$$
e_{\gamma,x_j} e_{x_i,\gamma} e_{x_j,\gamma} e_{\gamma,x_i}
$$
is an irregular expression.
\end{example}\qed

Let $\mathbf{Irr}$   be
the {\textit{left ideal}} of $\mathbf{U}(gl(m_0|m_1+n))$ generated by the set of
irregular expressions.

\begin{remark}
The action
of any element of $\mathbf{Irr}$ on the subalgebra $\mathbb C[M_{n,d}] \subset \mathbb{C}[M_{m_0|m_1+n,d}]$ - via the representation $\varrho$ -
is identically zero.
\end{remark}

\begin{proposition}(\cite{Brini3-BR}, \cite{BriUMI-BR})
The sum ${\mathbf{U}}(gl(0|n)) + \mathbf{Irr}$ is a direct sum of vector subspaces of $\mathbf{U}(gl(m_0|m_1+n)).$
\end{proposition}

We come now to one of the main notions of the virtual method.

The {\textit{virtual algebra}} $Virt(m_0+m_1,n)$ is the subalgebra
$$
Virt(m_0+m_1,n) = \mathbf{U}(gl(0|n)) \oplus \mathbf{Irr} \subset {\mathbf{U}}(gl(m_0|m_1+n)).
$$

The proof of the following proposition is immediate from the definitions.
\begin{proposition}
The left ideal  $\mathbf{Irr}$ of ${\mathbf{U}}(gl(m_0|m_1+n))$
is a two sided ideal of $Virt(m_0+m_1,n).$
\end{proposition}

The {\textit{Capelli devirtualization epimorphism}} is the projection
$$
\mathfrak{p} : Virt(m_0+m_1,n) = \mathbf{U}(gl(0|n)) \oplus \mathbf{Irr} \twoheadrightarrow \mathbf{U}(gl(0|n)) = \mathbf{U}(gl(n))
$$
with $Ker(\mathfrak{p}) = \mathbf{Irr}.$

\begin{example}
Let $x \in L, \ \alpha \in A_0.$ The element
$$
e_{x,\alpha}e_{\alpha,x} = - e_{\alpha,x}e_{x,\alpha} + e_{x,x} + e_{\alpha,\alpha}
$$
belongs to the virtual algebra $Virt(m_0+m_1,n)$ and
$$
\mathfrak{p} \big( e_{x,\alpha}e_{\alpha,x} \big) = e_{x,x} \in \mathbf{U}(gl(n)).
$$
\end{example}

\begin{example}\label{first Capelli}
 Let $x, y \in L, \ \alpha \in A_0.$ Then
\begin{align*}
e_{y \alpha} e_{x \alpha} e_{\alpha x} e_{\alpha y}
 &= -e_{y \alpha} e_{\alpha x} e_{x \alpha} e_{\alpha y}
+e_{y \alpha} e_{x x} e_{\alpha
y} +e_{y \alpha} e_{\alpha \alpha} e_{\alpha y}
\\ &= + e_{y \alpha} e_{\alpha x} e_{\alpha y} e_{x \alpha}
-e_{y \alpha} e_{\alpha x} e_{x
y}
\\ &\hphantom{{}=} -e_{x x} e_{\alpha y} e_{y \alpha}
+e_{x x} e_{y y} -e_{x x}
e_{\alpha \alpha}
\\ &\hphantom{{}=} +e_{y \alpha} e_{\alpha y} e_{\alpha \alpha}
+e_{y \alpha} e_{\alpha y}
\\ &= + e_{y \alpha} e_{\alpha x} e_{\alpha y} e_{x \alpha}
-e_{\alpha x} e_{y \alpha}  e_{x y} - e_{y x}e_{x y}
\\ &\hphantom{{}=} -e_{x x} e_{\alpha y} e_{y \alpha}
+e_{x x} e_{y y} -e_{x x}
e_{\alpha \alpha}
\\ &\hphantom{{}=} +e_{y \alpha} e_{\alpha y} e_{\alpha \alpha}
+ e_{\alpha y} e_{y \alpha} + e_{y y} + e_{\alpha \alpha} \in {\mathbf{U}}(gl(m_0|m_1+n)).
\end{align*}
Therefore
$$
e_{y \alpha} e_{x \alpha} e_{\alpha x} e_{\alpha y} \in Virt(m_0+m_1,n)
$$
and
$$
\mathfrak{p} \big( e_{y \alpha} e_{x \alpha} e_{\alpha x} e_{\alpha y} \big) =
-e_{y x} e_{x y} +e_{x x} e_{y y} +e_{y y} \in \mathbf{U}(gl(n)).
$$ \qed
\end{example}

Any element in $\textbf{M} \in Virt(m_0+m_1,n)$ defines an element in
$\textbf{m} \in \mathbf{U}(gl(n))$ - via the map $\mathfrak{p}$ -
 and $\textbf{M}$ is called a \textit{virtual
presentation} of $\textbf{m}$.

Since the map $\mathfrak{p}$  a surjection, any element
$\mathbf{m} \in \mathbf{U}(gl(n))$ admits several virtual
presentations. In the sequel, we even take virtual presentations
as the \emph{true definition} of special elements in $\mathbf{U}(gl(n)),$ and this method will turn out to be quite effective.

\begin{example}{\textbf{(A virtual presentation of the Capelli determinant)}}\label{Capelli determinants-BR}
As a generalization of  Example \ref{first Capelli}, we describe a ``\emph{monomial}'' virtual presentation in $Virt(m_0+m_1,n)$
of the classical Capelli determinant in $\mathbf{U}(gl(n)).$

Let $\alpha \in A_0$. The monomial element
\begin{equation}\label{C}
C = e_{x_n,\alpha} \cdots e_{x_2,\alpha} e_{x_1,\alpha} \cdot e_{\alpha, x_1} e_{\alpha, x_2}
 \cdots e_{\alpha, x_n} \in {\mathbf{U}}(gl(m_0|m_1+n))
\end{equation}
belongs to the virtual algebra $Virt(m_0|m_1+n)$.
The image of the element $C$ under the Capelli devirtualization epimorphism $\mathfrak{p}$
equals the {\it{column determinant}}\footnote{The symbol $\mathbf{cdet}$
denotes the column determinat of a matrix $A = [a_{ij}]$ with noncommutative entries:
$\mathbf{cdet} (A) = \sum_{\sigma} \ (-1)^{|\sigma|}  \ a_{\sigma(1), 1}a_{\sigma(2), 2} \cdots a_{\sigma(n), n}.$}
 $$
\mathbf{H}_n(n) = \mathbf{cdet}\left(
 \begin{array}{cccc}
 e_{x_1,x_1}+(n-1) & e_{x_1,x_2} & \ldots  & e_{x_1,x_n} \\
 e_{x_2,x_1} & e_{x_2,x_2}+(n-2) & \ldots  & e_{x_2,x_n}\\
 \vdots  &    \vdots                            & \vdots &  \\
e_{x_n,x_1} & e_{x_n,x_2} & \ldots & e_{x_n,x_n}\\
 \end{array}
 \right) \in \mathbf{U}(gl(n)).
 $$
This result is a special case of the result that we called the  ``Laplace expansion for Capelli rows'' ( \cite{BRT-BR} Theorem $2$,
\cite{Bri-BR} Theorem $6.3$). A sketchy proof of it
can also be found in Koszul \cite{Koszul-BR}. \qed
\end{example}

The next results will play a crucial role in the study of {\textit{central}} elements of $\mathbf{U}(gl(n))$.
\begin{proposition}\label{rappresentazione aggiunta-BR}
For every $e_{x_i, x_j} \in gl(n) \subset gl(m_0|m_1+n)$,   let $ad(e_{x_i, x_j})$ denote its adjoint action
on  $Virt(m_0+m_1,n)$; the ideal $\mathbf{Irr}$ is $ad(e_{x_i, x_j})-$invariant. Then
\begin{equation}
\mathfrak{p} \left( ad(e_{x_i, x_j})( \mathbf{m} ) \right) =  ad(e_{x_i, x_j}) \left( \mathfrak{p} ( \mathbf{m} ) \right) ,
\qquad  \mathbf{m} \in Virt(m_0+m_1,n).
\end{equation}
\end{proposition}

\begin{corollary}\label{invarianti virtuali}
The Capelli epimorphism image of an element of $Virt(m_0|m_1+n)$ that is an invariant for the adjoint action of
$gl(n)$ is in the center $\boldsymbol{\zeta}(n)$ of $\mathbf{U}(gl(n))$.
\end{corollary}

\begin{example}
Recall that
$$ad(e_{x_i, x_j})\left( e_{x_h, \alpha} \right) = \delta_{j h}e_{x_i, \alpha},$$
$$ad(e_{x_i, x_j})\left( e_{\alpha, x_k} \right) = - \delta_{k i}e_{\alpha, x_j},$$
for every virtual symbol $\alpha$, and that $ad(e_{x_i, x_j})$ acts as a derivation, for every $i, j = 1,2, \ldots, n.$

The  monomial $C$ of Example \ref{Capelli determinants-BR}, Eq.(\ref{C}) is annihilated by
$ad(e_{x_i, x_j}), \ i \neq j,$ by skew-symmetry. Furthermore, $ad(e_{x_i, x_i})\left( C \right) = C - C = 0, \ i = 1, 2, \ldots, n;$
hence, $C$ is an invariant for the adjoint action of $gl(n)$.

Since
$\mathfrak{p} \left( C \right) = \mathbf{H}_n(n),$ the Capelli determinant $\mathbf{H}_n(n)$ is central in $\mathbf{U}(gl(n)),$
by Corollary \ref{invarianti virtuali}. \qed
\end{example}

\subsubsection{The action of  $Virt(m_0+m_1,n)$ on the subalgebra $\mathbb{C}[M_{n,d}]$}

From the representation-theoretic point of view, the core of the {\textit{method of virtual variables}}
lies in the following result.

\begin{theorem}\label{Capelli epimorphism}
The action of $Virt(m_0+m_1,n)$ leaves invariant the
subalgebra  $\mathbb{C}[M_{n,d}] \subseteq
\mathbb{C}[M_{m_0|m_1+n,d}],$ and, therefore, the action of $Virt(m_0+m_1,n)$ on $\mathbb{C}[M_{n,d}]$
is well defined.
Furthermore, for every $\mathbf{v} \in Virt(m_0+m_1,n)$, its action on $\mathbb{C}[M_{n,d}]$ equals
the action of $\mathfrak{p}(\mathbf{v}) \in \mathbf{U}(gl(n)).$
\end{theorem}

Therefore, instead of studying the action of an element in $\mathbf{U}(gl(n))$, one can study the action of a virtual presentation of it in
$Virt(m_0|m_1+n)$. The advantage of virtual presentations is that they
are frequently of monomial form,   admit quite transparent interpretations and are  much  easier to be dealt with
(see, e.g. \cite{Brini1-BR}, \cite{Brini2-BR}, \cite{BRT-BR}, \cite{Bri-BR}, \cite{BriUMI-BR}).

A prototypical instance of this method is provided by the celebrated \emph{Capelli identity} \cite{Cap1-BR}, \cite{Weyl-BR},
\cite{Howe-BR}, \cite{HU-BR}, \cite{Umeda-BR}. From Example \ref{Capelli determinants-BR}, it follows that the action of the Capelli
determinant $\mathbf{H}_n(n)$
on a form $f \in \mathbb{C}[M_{n,d}]$ is the same as the action of its monomial virtual presentation, and this leads to
a few lines proof of the identity \cite{BRT-BR}, \cite{BriUMI-BR}.

\subsubsection{Balanced monomials as elements of the virtual algebra $Virt(m_0+m_1,n)$}\label{balanced monomial}

In order to make the virtual variables method  effective, we need to exhibit a class of nontrivial elements that belong
to $Virt(m_0+m_1,n)$.

A quite relevant class of such elements is provided
by \textit{balanced monomials}.

In plain words, a balanced monomial is product of two or more factors  where the
rightmost one  \textit{annihilates}
the $k$ proper symbols $ x_{j_1}, \ldots, x_{j_k}$ and
\textit{creates} some virtual symbols;
 the leftmost one  \textit{annihilates} all the virtual symbols
and \textit{creates} the $k$ proper symbols $ x_{i_1}, \ldots, x_{i_k}$;
between these two factors, there might be further factors that annihilate
 and create  virtual symbols only.

In a formal way,  balanced monomials are  elements of the algebra ${\mathbf{U}}(gl(m_0|m_1+n))$
 of the forms:
\begin{itemize}\label{defbalanced monomials-BR}
\item $e_{x_{i_1},\gamma_{p_1}} \cdots e_{x_{i_k},\gamma_{p_k}} \cdot
e_{\gamma_{p_1},x_{j_1}} \cdots e_{\gamma_{p_k},x_{j_k}},$
\item
$e_{x_{i_1},\theta_{q_1}} \cdots e_{x_{i_k},\theta_{q_k}} \cdot
e_{\theta_{q_1},\gamma_{p_1}} \cdots e_{\theta_{q_k},\gamma_{p_k}} \cdot
e_{\gamma_{p_1},x_{j_1}} \cdots e_{\gamma_{p_k},x_{j_k}},$
\item and so on,
\end{itemize}
where
$x_{i_1}, \ldots, x_{i_k}, x_{j_1}, \ldots, x_{j_k} \in L,$
i.e., the $x_{i_1}, \ldots, x_{i_k}, x_{j_1}, \ldots, x_{j_k}$ are $k$
proper  symbols.

The next result is the (superalgebraic) formalization of the argument developed by Capelli in
\cite{Cap4-BR}, CAPITOLO I, §X.Metodo delle variabili ausiliarie, page $55$ ff.

\begin{proposition}\emph{(\cite{Brini1-BR}, \cite{Brini2-BR}, \cite{BRT-BR}, \cite{Bri-BR}, \cite{BriUMI-BR})}
Every balanced monomial belongs to $Virt(m_0+m_1,n)$. Hence
its image under the Capelli epimorphism $\mathfrak{p}$ belongs to $\mathbf{U}(gl(n)).$
\end{proposition}

In plain words, the action of a balanced monomial on the subalgebra  $\mathbb{C}[M_{n,d}]$ equals
the action of a suitable element of $\mathbf{U}(gl(n)).$

\subsection{Two special classes of elements in $Virt(m_0+m_1,n)$ and their images in ${\mathbf{U}}(gl(n))$}\label{Capbit}

We will introduce two classes of remarkable elements of the enveloping algebra ${\mathbf{U}}(gl(n))$, that
we call {\textit{Capelli bitableaux}} and {\textit{Young-Capelli bitableaux}}, respectively.

Capelli bitableaux are the analogues in ${\mathbf{U}}(gl(n))$ of bitableaux in the polynomial algebra
$\mathbb{C}[M_{n,d}]$, as well as Young-Capelli bitableaux are the analogues in ${\mathbf{U}}(gl(n))$ of
right symmetrized bitableaux. Besides this analogy, their meaning lies deeper, as we shall see in
Section \ref{sect Kosz}.

\subsubsection{Bitableaux monomials in ${\mathbf{U}}(gl(m_0+m_1,n))$}\label{Bitableaux monomials}
Let $S$ and $T$ be two Young tableaux of same shape $\lambda \vdash h$ on
the  alphabet $ A_0 \cup A_1 \cup L$:

\begin{equation}\label{bitableaux}
S = \left(
\begin{array}{llllllllllllll}
z_{i_1}  \ldots    \ldots     \ldots     z_{i_{\lambda_1}}     \\
z_{j_1}   \ldots  \ldots               z_{j_{\lambda_2}} \\
 \ldots  \ldots   \\
z_{s_1} \ldots z_{s_{\lambda_p}}
\end{array}
\right), \qquad
T = \left(
\begin{array}{llllllllllllll}
z_{h_1}  \ldots    \ldots     \ldots     z_{h_{\lambda_1}}    \\
z_{k_1}   \ldots  \ldots               z_{k_{\lambda_2}} \\
 \ldots  \ldots   \\
z_{t_1} \ldots z_{t_{\lambda_p}}
\end{array}
\right).
\end{equation}

To the pair $(S,T)$, we associate the {\it{bitableau monomial}}:
\begin{equation}\label{BitMon}
e_{S,T} =
e_{z_{i_1}, z_{h_1}}\cdots e_{z_{i_{\lambda_1}}, z_{h_{\lambda_1}}}
e_{z_{j_1}, z_{k_1}}\cdots e_{z_{j_{\lambda_2}}, z_{k_{\lambda_2}}}
 \cdots  \cdots
e_{z_{s_1}, z_{t_1}}\cdots e_{z_{s_{\lambda_p}}, z_{t_{\lambda_p}}}
\end{equation}
in ${\mathbf{U}}(gl(m_0|m_1+n)).$

By expressing the Young tableaux $S, T$ in the functional form (see subsection \ref{comb Young tab}):
$$
S : \underline{h} \rightarrow A_0 \cup A_1 \cup L, \quad T : \underline{h} \rightarrow A_0 \cup A_1 \cup L,
$$
the bitableau monomial $e_{S,T}$ of Eq. (\ref{BitMon}) becomes:
$$
e_{S,T} = e_{S(1),T(1)}e_{S(2),T(2)} \cdots e_{S(h),T(h)}.
$$

Let us denote by $\alpha_1, \ldots, \alpha_p \in A_0$, $\beta_1, \ldots, \beta_{\lambda_1} \in A_1$
two \emph{arbitrary} families of \emph{mutually distinct positive and negative virtual symbols}, respectively (see Remark \ref{symbols}).
Set
\begin{equation}\label{Deruyts and Coderuyts}
D_{\lambda}^* = \left(
\begin{array}{llllllllllllll}
\beta_1  \ldots    \ldots     \ldots     \beta_{\lambda_1}     \\
\beta_1   \ldots  \ldots               \beta_{\lambda_2} \\
 \ldots  \ldots   \\
\beta_1 \ldots \beta_{\lambda_p}
\end{array}
\right), \qquad
C_{\lambda}^* = \left(
\begin{array}{llllllllllllll}
\alpha_1  \ldots    \ldots     \ldots     \alpha_1    \\
\alpha_2   \ldots  \ldots               \alpha_2 \\
 \ldots  \ldots   \\
\alpha_p \ldots \alpha_p
\end{array} \right).
\end{equation}

The tableaux of  kind (\ref{Deruyts and Coderuyts}) are called  {\it{virtual Deruyts and Coderuyts}} tableaux
of shape $\lambda,$ respectively.

\subsubsection{Capelli bitableaux and Young-Capelli bitableaux}\label{Capbit sub}
Given a pair of Young tableaux  $S, T$ of the same shape $\lambda$ on the proper alphabet $L$, consider the elements
\begin{equation}
e_{S,C_{\lambda}^*} \ e_{C_{\lambda}^*,T} \in {\mathbf{U}}(gl(m_0|m_1+n)),
\end{equation}
\begin{equation}
e_{S,C_{\lambda}^*} \ e_{C_{\lambda}^*,D_{\lambda}^*}  \   e_{D_{\lambda}^*,T} \in {\mathbf{U}}(gl(m_0|m_1+n)).
\end{equation}

Since elements (\ref{determinantal}) and (\ref{rightYoungCapelli}) are balanced monomials in
${\mathbf{U}}(gl(m_0|m_1+n))$, then they belong to the subalgebra $Virt(m_0+m_1,n)$ (Section \ref{balanced monomial}).

Hence, we can consider their images in ${\mathbf{U}}(gl(n))$ with respect to the Capelli epimorphism $\mathfrak{p}$.

We set
\begin{equation}\label{determinantal}
 [S|T] = \mathfrak{p} \Big( e_{S,C_{\lambda}^*} \ e_{C_{\lambda}^*,T}    \Big)   \in {\mathbf{U}}(gl(n)),
\end{equation}
and call the element $[S|T]$ a {\textit{Capelli bitableau}}.

We set
\begin{equation}\label{rightYoungCapelli}
 [S| \fbox{$T$}]  = \mathfrak{p} \Big( e_{S,C_{\lambda}^*} \ e_{C_{\lambda}^*,D_{\lambda}^*}  \   e_{D_{\lambda}^*,T} \Big)
          \in {\mathbf{U}}(gl(n)).
\end{equation}
and call the element $[S| \fbox{$T$}] $ a {\textit{Young-Capelli bitableau}}.

\begin{remark}

The elements defined in \emph{(\ref{determinantal})} and \emph{(\ref{rightYoungCapelli})}  do not
depend on the choice of the virtual Deruyts and Coderuyts tableaux $D_{\lambda}^*$
and $C_{\lambda}^*$.
\end{remark}\qed

The next result will play a crucial role subsection \ref{sect YC Bit} below.
In plain words, it states that  {\textit{Young-Capelli bitableaux} expand into
{\textit{Capelli bitableaux} in the enveloping algebra $\mathbf{U}(gl(n))$ just in the same
formal way as {\textit{right symmetrized bitableaux} expand into
{\textit{bitableaux} in the polynomial algebra ${\mathbb C}[M_{n,d}]$ (subsection \ref{symm bit}).

\begin{proposition}\label{YC bit and C bit}
Let $S, T$ be Young tableaux, $sh(S) = sh(T)$. The  following identity holds
in the enveloping algebra $\mathbf{U}(gl(n))$:
$$
[S|\fbox{$T$})] = \sum_{\overline{T}} \ [S|\overline{T}],
$$
where the sum is extended over {\textit{all}} $\overline{T}$ column permuted of $T$ (hence, repeated entries in
a column give rise to multiplicities).
\end{proposition}
The proof easily follows from the definitions, by applying the commutator identities
in the superalgebra ${\mathbf{U}}(gl(m_0|m_1+n))$.

\begin{example}(cfr. Example \ref{ex symm})
\begin{align*}
\left[
\begin{array}{cc}
 x_1 & x_3 \\  x_2 & x_4
\end{array}
\right| \left.
\fbox{$
\begin{array}{cc}
x_1 & x_2 \\ x_1 & x_3
\end{array}
$} \
\right]
& =
\left[
\begin{array}{cc}
 x_1 & x_3 \\  x_2 & x_4
\end{array}
\right| \left.
\begin{array}{cc}
x_1 & x_2 \\ x_1 & x_3
\end{array}
\right]
+
\left[
\begin{array}{cc}
 x_1 & x_3 \\  x_2 & x_4
\end{array}
\right| \left.
\begin{array}{cc}
x_1 & x_2 \\ x_1 & x_3
\end{array}
\right]
\\
&  +
\left[
\begin{array}{cc}
 x_1 & x_3 \\  x_2 & x_4
\end{array}
\right| \left.
\begin{array}{cc}
x_1 & x_3 \\ x_1 & x_2
\end{array}
\right]
+
\left[
\begin{array}{cc}
 x_1 & x_3 \\  x_2 & x_4
\end{array}
\right| \left.
\begin{array}{cc}
x_1 & x_3 \\ x_1 & x_2
\end{array}
\right]
\\
& =
2  \left[
\begin{array}{cc}
 x_1 & x_3 \\  x_2 & x_4
\end{array}
\right| \left.
\begin{array}{cc}
x_1 & x_2 \\ x_1 & x_3
\end{array}
\right]
+
2  \left[
\begin{array}{cc}
 x_1 & x_3 \\  x_2 & x_4
\end{array}
\right| \left.
\begin{array}{cc}
x_1 & x_3 \\ x_1 & x_2
\end{array}
\right].
\end{align*}
\end{example}\qed

\section{The {\textit{bitableaux correspondence}} isomorphism $\mathcal{K}^{-1}$ and the Koszul map $\mathcal{K}$} \label{sect Kosz}

\subsection{The BCK theorem}

Our next aim is to describe an extremely relevant pair of (mutually inverse) vector space isomorphisms
between the polynomial algebra of  forms ${\mathbb C}[M_{n,n}]$
and the universal enveloping algebra  $\mathbf{U}(gl(n))$.

In order to do this, it is worth to simplify the notation in the following way:
\begin{itemize}
\item
we will write $i$ in place of $x_i$ and $e_{ij}$ in place of $e_{x_i, x_j}$;

\item
consistently, we set $L = P = \underline{n} = \{1, 2, \ldots, n \}.$
\end{itemize}

The main advantage of this convention is that it allows us to write bitableaux in
${\mathbb C}[M_{n,n}]$ and Capelli bitableaux in $\mathbf{U}(gl(n))$ as elements associated
to pairs of Young tableaux on the {\textit{same}} alphabet.

More specifically,
given a shape (partition) $\lambda$ with $\lambda_1 \leq n$, to any
pair of Young tableaux $S, T$ on the alphabet $\underline{n} = \{1, 2, \ldots, n \}$ and
of the same shape $sh(S) = sh(T) = \lambda$, one associates the (determinantal) bitableau $(S|T) \in {\mathbb C}[M_{n,n}]$,
and the  Capelli bitableau $[S|T] \in \mathbf{U}(gl(n))$.

\begin{theorem} {\bf{(The BCK theorem)}}  \label{KBT corr}
The ``bitableaux correspondence'' map
\begin{equation}\label{operator B}
\mathcal{K}^{-1} : (S|T) \mapsto [S|T]
\end{equation}
uniquely defines a linear isomorphism
$$
\mathcal{K}^{-1} : {\mathbb C}[M_{n,n}] \cong \mathbf{Sym}(gl(n)) \rightarrow \mathbf{U}(gl(n)).
$$
Furthermore, this isomorphism is the inverse of the Koszul map
$$
\mathcal{K} : \mathbf{U}(gl(n)) \rightarrow {\mathbb C}[M_{n,n}] \cong \mathbf{Sym}(gl(n))
$$
introduced by J.-L. Koszul in \emph{\cite{Koszul-BR}}.
\end{theorem}

Eq. (\ref{operator B}) indeed defines a {\emph{linear}} operator since
bitableaux in ${\mathbb C}[M_{n,n}]$ and Capelli bitableaux
in $\mathbf{U}(gl(n))$ are ruled by the \textbf{same} \emph{straightening laws} (see \cite{Brini3-BR}, Proposition $7$).

The linear isomorphism $\mathcal{K}^{-1}$ was introduced in
\cite{Brini4-BR}, Theorem $1$.  The fact that $\mathcal{K}^{-1}$ and
$\mathcal{K}$ are inverse of each other was proved in
\cite{Brini4-BR}, Theorem $2$ (see also, \cite{BriUMI-BR}).

\subsection{Right symmetrized  bitableaux and Young-Capelli bitableaux}\label{sect YC Bit}

The "bitableaux correspondence" and the Koszul isomorphisms behave well with
respect to right symmetrized  bitableaux
$$
(S|\fbox{$T$}) \in {\mathbb C}[M_{n,n}]
$$
and {\textit{Young-Capelli bitableaux}}
$$
[S| \fbox{$T$} ]
=
\mathfrak{p} \left( e_{S C^*_{\lambda}} e_{C^*_{\lambda} D^*_{\lambda}} e_{D^*_{\lambda} T}c \right)
 \in \mathbf{U}(gl(n)).
$$

In plain words, any Young-Capelli bitableaux $[S| \fbox{$T$} ]$ is the image -
with respect to the linear operator $\mathcal{K}^{-1}$ - of the right symmetrized bitableaux $(S|\fbox{$T$})$.

\begin{theorem}\label{image symm} We have:
\begin{align*}
&\mathcal{K}^{-1} : (S|\fbox{$T$}) \mapsto [S| \fbox{$T$} ],
\\
&\mathcal{K} : [S| \fbox{$T$} ] \mapsto (S|\fbox{$T$}).
\end{align*}
\end{theorem}

\begin{proof}
Indeed, we have:
\begin{align*}
\mathcal{K}^{-1} \Big( (S|\fbox{$T$}) \Big) =& \mathcal{K}^{-1} \Big( \sum_{\overline{T}} \ (S|\overline{T} ) \Big)
\\
=& \sum_{\overline{T}} \ [S|\overline{T} ],
\end{align*}
where the sum is extended over {\textit{all}} $\overline{T}$ column permuted of $T$.

By Proposition \ref{YC bit and C bit}, the last summation equals the Young-Capelli
bitableaux $[S| \fbox{$T$} ]$.
\end{proof}

By Theorem \ref{image symm} and Theorem  \ref{GC basis}, we have:

\begin{theorem}\label{Y-C basis}
Let $h \in {\mathbb N}.$ The set of Young-Capelli bitableaux
$$
 \bigcup_{k = 0}^h  \ \Big\{ \ [S|\fbox{$T$}]; \ S, T \ standard, \ sh(S) = sh(T) = \lambda \vdash k, \ \lambda_1 \leq n \Big\}
$$
is a basis of the filtration element $\mathbf{U}(gl(n))^{(h)}$.
\end{theorem}

\begin{remark}\label{action}
The basis elements
$$
\Big\{ \ [S|\fbox{$T$}]; \ S, T \ standard, \ sh(S) = sh(T) = \lambda \vdash k, \ \lambda_1 \leq n \Big\}
$$
act in a quite remarkable way on Gordan-Capelli basis elements
$$
\Big\{ \ (U|\fbox{$V$}); \ U, V\ standard, \ sh(U) = sh(V) = \mu \vdash h, \ \mu_1 \leq n \Big\}.
$$
Indeed, we have:

\begin{itemize}
\item [--]
If $h < k$, the action is zero.

\item [--]
If $h = k$ and $\lambda \neq \mu$, the action is zero.

\item [--]
If $h = k$ and $\lambda = \mu$, the action is nondegenerate triangular
(with respect to a suitable linear order on standard tableaux of the same shape).
\end{itemize}
See \emph{\cite{Brini2-BR}} and  \emph{\cite{Bri-BR}}, \emph{Theorem} $10.1$. \qed
\end{remark}

\subsection{Column Capelli bitableaux in $\mathbf{U}(gl(n))$}\label{column Capelli sec}

A column Capelli bitableau in $\mathbf{U}(gl(n))$ is a Capelli bitableau $[S|T]$, where $S$ and $T$ are column Young
tableaux of the same depth.

Although column Capelli bitableaux are far from being ``monomials'' in $\mathbf{U}(gl(n))$, they play the same role that
column  bitableaux -- signed monomials -- play in the polynomial algebra ${\mathbb C}[M_{n,n}]$. Specifically,
Capelli bitableaux and Young-Capelli bitableaux expand into column Capelli bitableaux just in the same way as
bitableaux and right symmetrized bitableaux expand into column  bitableaux in the polynomial algebra ${\mathbb C}[M_{n,n}]$.

\begin{remark}
The column Capelli bitableau $[i|j]$ of depth $h = 1$  equals the generator
$e_{i,j}$ of the algebra $\mathbf{U}(gl(n))$,  $i, j = 1,2, \ldots, n$, $i, j = 1,2, \ldots, n$.
Indeed
$$
[i|j] = \mathfrak{p} \left[ e_{i, \alpha}  e_{\alpha, j} \right] =
\mathfrak{p} \left[  - e_{\alpha, j} e_{i, \alpha} +  e_{i, j} + \delta_{i, j} e_{\alpha, \alpha}  \right]
= e_{i j}.
$$
\end{remark}

Since column  bitableaux in the polynomial algebra ${\mathbb C}[M_{n,n}]$ are signed \emph{commutative} monomials,
then column Capelli bitableaux are invariant with respect to permutations of their rows,
that is
$$
\left[
\begin{array}{c}
i_1\\  i_2 \\ \vdots \\ i_h
\end{array}
\right| \left.
\begin{array}{c}
j_1\\  j_2 \\ \vdots \\ j_h
\end{array}
\right]
=
\left[
\begin{array}{c}
i_{\sigma(1)}\\  i_{\sigma(2)} \\ \vdots \\ i_{\sigma(h)}
\end{array}
\right| \left.
\begin{array}{c}
j_{\sigma(1)}\\  j_{\sigma(2)} \\ \vdots \\ j_{\sigma(h)}
\end{array}
\right]
$$
for every $\sigma \in \mathbf{S}_h.$

Let us denote by ${\mathbb C}_h[M_{n,n}]$ the homogeneous component of degree $h \in \mathbb{N}$
of the polynomial algebra ${\mathbb C}[M_{n,n}]$ and denote $\mathbf{U}(gl(n))^{(h)}$ the $h-$th
filtration element of the enveloping algebra $\mathbf{U}(gl(n)).$

\begin{corollary}\label{KBT filtr}
The bitableaux correspondence isomorphism $\mathcal{K}^{-1}$ and the Koszul isomorphisms  $\mathcal{K}$
induce, by restriction, a pair of mutually inverse isomorphisms
$$
\mathcal{K}^{-1} : \bigoplus_{k = 0}^h \ {\mathbb C}_k[M_{n,n}] \ \rightarrow \ \mathbf{U}(gl(n))^{(h)}
$$
and
$$
\mathcal{K} : \mathbf{U}(gl(n))^{(h)}\ \rightarrow \ \bigoplus_{k = 0}^h \ {\mathbb C}_k[M_{n,n}] .
$$
\end{corollary}
The preceding assertion can be regarded as a sharpened version of the PBW Theorem for $\mathbf{U}(gl(n)).$

\subsubsection{Devirtualization of column Capelli bitableaux in $\mathbf{U}(gl(n))$}\label{column Cap dev}

Given any column Capelli bitableau,  \emph{devirtualized
expressions} of it as an element of $\mathbf{U}(gl(n))$ can be easily obtained by means of iterations of
the following identities.

\begin{proposition} \label{devirtualization column}
In the enveloping algebra $\mathbf{U}(gl(n))$,
we have:

$$
\left[
\begin{array}{c}
i_1\\  i_2 \\ \vdots \\ i_{h-1} \\ i_h
\end{array}
\right| \left.
\begin{array}{c}
j_1 \\ j_2 \\ \vdots\\ j_{h-1} \\ j_h
\end{array}
\right] =
$$

$$
(-1)^{h-1} \ e_{i_1, j_1} \ \left[
\begin{array}{c}
i_2\\  \vdots \\ i_{h-1} \\ i_h
\end{array}
\right| \left.
\begin{array}{c}
j_2 \\ \vdots \\  j_{h-1} \\ j_h
\end{array}
\right]  \ + \ (-1)^{h-2} \ \sum_{k=2}^h \delta_{i_k,j_1} \ \left[
\begin{array}{c}
 i_1 \\  \vdots  \\ i_{k-1}  \\  i_{k+1} \\ \vdots   \\ i_h
\end{array}
\right| \left.
\begin{array}{c}
j_k \\ \vdots  \\ j_{k-1}  \\ j_{k+1} \\ \vdots \\ j_h
\end{array}
\right] =
$$

$$
(-1)^{h-1} \ \left[
\begin{array}{c}
i_1\\  i_2 \\ \vdots \\ i_{h-1}
\end{array}
\right| \left.
\begin{array}{c}
j_1 \\ j_2 \\ \vdots\\  j_{h-1}
\end{array}
\right] \ e_{i_h, j_h} \ + \ (-1)^{h-2} \ \sum_{k=1}^{h-1}
\delta_{i_h,j_k} \ \left[
\begin{array}{c}
i_1\\  \vdots  \\ i_{k-1}  \\  i_{k+1} \\ \vdots  \\ i_k
\end{array}
\right| \left.
\begin{array}{c}
j_1 \\ \vdots  \\ j_{k-1}  \\ j_{k+1} \\  \vdots  \\ j_h
\end{array}
\right].
$$

\end{proposition}

\begin{proof}

By definition,
$$
\left[
\begin{array}{c}
i_1\\  i_2 \\ \vdots \\ i_{h-1} \\ i_h
\end{array}
\right| \left.
\begin{array}{c}
j_1 \\ j_2 \\ \vdots\\ j_{h-1} \\ j_h
\end{array}
\right] =
$$
\begin{align*}
=& \ \mathfrak{p} \big[ e_{{i_1},\alpha_1} e_{{i_2},\alpha_2}   \cdots e_{{i_{h-1}},\alpha_{h-1}} e_{{i_{h}},\alpha_h} \cdot
e_{\alpha_1, {j_1}} e_{\alpha_2, {j_2}}   \cdots e_{\alpha_{h-1}, {j_{h-1}}} e_{\alpha_h, {j_{h}}} \big] =                                            \\
=& \ \mathfrak{p} \big[
  - e_{{i_1},\alpha_1} e_{{i_2},\alpha_2}   \cdots e_{{i_{h-1}},\alpha_{h-1}} e_{\alpha_1, {j_1}}
 e_{{i_{h}},\alpha_h} \cdot e_{\alpha_2, {j_2}}   \cdots e_{\alpha_{h-1}, {j_{h-1}}} e_{\alpha_h, {j_{h}}}
\\
&\phantom{\ \mathfrak{p} \big[} + e_{{i_1},\alpha_1} e_{{i_2},\alpha_2}   \cdots e_{{i_{h-1}},\alpha_{h-1}}
  \cdot \delta_{i_h, j_1} e_{\alpha_1, \alpha_h} e_{\alpha_2, {j_2}}   \cdots  e_{\alpha_{h-1}, {j_{h-1}}} e_{\alpha_h, {j_{h}}} \big] =
\\
=& \ \mathfrak{p} \big[
  - e_{{i_1},\alpha_1} e_{{i_2},\alpha_2}   \cdots e_{{i_{h-1}},\alpha_{h-1}} e_{\alpha_1, {j_1}}
 e_{{i_{h}},\alpha_h} \cdot e_{\alpha_2, {j_2}}   \cdots e_{\alpha_{h-1}, {j_{h-1}}} e_{\alpha_h, {j_{h}}}
\\
&\phantom{\ \mathfrak{p} \big[} +
e_{{i_1},\alpha_1} e_{{i_2},\alpha_2}   \cdots e_{{i_{h-1}},\alpha_{h-1}}
  \cdot \delta_{i_h, j_1} e_{\alpha_2, {j_2}}   \cdots  e_{\alpha_{h-1}, {j_{h-1}}} e_{\alpha_1, {j_{h}}} \big].
\end{align*}
Notice that
$$
 \delta_{i_h, j_1} \
e_{{i_1},\alpha_1} e_{{i_2},\alpha_2}   \cdots e_{{i_{h-1}},\alpha_{h-1}}
  \cdot  e_{\alpha_2, {j_2}}   \cdots  e_{\alpha_{h-1}, {j_{h-1}}} e_{\alpha_1, {j_{h}}}  =
$$
$$
 \delta_{i_h, j_1} \ (-1)^{h - 2} \
e_{{i_1},\alpha_1} e_{{i_2},\alpha_2}   \cdots e_{{i_{h-1}},\alpha_{h-1}} \cdot
 e_{\alpha_1, {j_{h}}}    e_{\alpha_2, {j_2}}   \cdots  e_{\alpha_{h-1}, {j_{h-1}}}
$$
as elements of the algebra ${\mathbf{U}}(gl(m_0|m_1+n)).$

Therefore, the summand
$$
\mathfrak{p} \big[ e_{{i_1},\alpha_1} e_{{i_2},\alpha_2}   \cdots e_{{i_{h-1}},\alpha_{h-1}}
  \cdot \delta_{i_h, j_1} e_{\alpha_2, {j_2}}   \cdots  e_{\alpha_{h-1}, {j_{h-1}}} e_{\alpha_1, {j_{h}}} \big]
$$
equals
$$(-1)^{h - 2} \
\delta_{i_h, j_1} \
\left[
\begin{array}{c}
i_1\\  i_2 \\ \vdots \\ i_{h-1}
\end{array}
\right| \left.
\begin{array}{c}
j_h \\ j_2 \\ \vdots\\ j_{h-1}
\end{array}
\right].
$$
By repeating the above procedure of moving
left the element $e_{\alpha_1, {j_1}}$ - using the commutator identities in ${\mathbf{U}}(gl(m_0|m_1+n))$ -
we finally get
$$
\left[
\begin{array}{c}
i_1\\  i_2 \\ \vdots \\ i_{h-1} \\ i_h
\end{array}
\right| \left.
\begin{array}{c}
j_1 \\ j_2 \\ \vdots\\ j_{h-1} \\ j_h
\end{array}
\right] =
$$
\begin{align*}
=& \ \mathfrak{p} \big[ (-1)^{h - 1}
  e_{{i_1},\alpha_1} e_{\alpha_1, {j_1}}e_{{i_2},\alpha_2}   \cdots e_{{i_{h-1}},\alpha_{h-1}}
 e_{{i_{h}},\alpha_h} \cdot e_{\alpha_2, {j_2}}   \cdots e_{\alpha_{h-1}, {j_{h-1}}} e_{\alpha_h, {j_{h}}}
\\
&\phantom{\ \mathfrak{p} \big[} +
 \sum_{i = 0}^{h - 2} \ (-1)^i \
e_{{i_1},\alpha_1}    \cdots \delta_{i_{h-i}, j_1}\widehat{e_{{i_{h-i}},\alpha_{h-i}}}e_{\alpha_1,\alpha_{h-i}} \cdots
e_{{i_{h}},\alpha_{h}}
  \cdot  e_{\alpha_2, {j_2}}   \cdots   e_{\alpha_h, {j_{h}}} \big]
\\
=&  \ \mathfrak{p} \big[ (-1)^{h - 1}
  e_{{i_1},\alpha_1} e_{\alpha_1, {j_1}}e_{{i_2},\alpha_2}   \cdots e_{{i_{h-1}},\alpha_{h-1}}
 e_{{i_{h}},\alpha_h} \cdot e_{\alpha_2, {j_2}}   \cdots e_{\alpha_{h-1}, {j_{h-1}}} e_{\alpha_h, {j_{h}}}
\\
&\phantom{\ \mathfrak{p} \big[} +
\sum_{i = 0}^{h - 2} \ (-1)^i \
e_{{i_1},\alpha_1} \cdot\cdot \delta_{i_{h-i}, j_1} \cdots
e_{{i_{h}},\alpha_{h}}
  \cdot  e_{\alpha_2, {j_2}} \cdot\cdot  e_{\alpha_1,j_{h-i}}\cdot \cdot e_{\alpha_h, {j_{h}}} \big].
\end{align*}

Notice that the summand
$$
(-1)^i \ \delta_{i_{h-i}, j_1} \
e_{{i_1},\alpha_1} \cdot\cdot \delta_{i_{h-i}, j_1} \cdots
e_{{i_{h}},\alpha_{h}}
  \cdot  e_{\alpha_2, {j_2}} \cdot \cdot  e_{\alpha_1,j_{h-i}}\cdot \cdot e_{\alpha_h, {j_{h}}}
$$
equals
$$
(-1)^i \ \delta_{i_{h-i}, j_1} \ (-1)^{h-i-2} \times
$$
$$
e_{{i_1},\alpha_1}    \cdots \widehat{e_{i_{h-i},\alpha_{h-i}}} \cdots
e_{{i_{h}},\alpha_{h}}
  \cdot e_{\alpha_1,j_{h-i}}  e_{\alpha_2, {j_2}}   \cdots \widehat{e_{\alpha_{h-i}, {j_{h-i}}}} \dots  e_{\alpha_h, {j_{h}}}
$$
as elements of the algebra ${\mathbf{U}}(gl(m_0|m_1+n)).$

Hence
$$
\mathfrak{p} \left[ (-1)^i \ \delta_{i_{h-i}, j_1} \
e_{{i_1},\alpha_1}    \cdots \widehat{e_{{i_{h-i}},\alpha_{h-i}}}e_{\alpha_1,\alpha_{h-i}} \cdots
e_{{i_{h}},\alpha_{h}}
  \cdot  e_{\alpha_2, {j_2}}   \cdots   e_{\alpha_h, {j_{h}}} \right]
$$
equals
$$ (-1)^{h-2} \ \delta_{i_{h-i}, j_1} \
\left[
\begin{array}{c}
i_1\\  i_2 \\ \vdots \\  i_{h-i-1} \\\widehat{i_{h-i}} \\ i_{h-i+1} \\ i_h
\end{array}
\right| \left.
\begin{array}{c}
j_{h-i} \\ j_2 \\ \vdots \\  j_{h-i-1} \\ \widehat{j_{h-i}} \\ j_{h-i+1} \\ j_h
\end{array}
\right].
$$
Furthermore
$$
\mathfrak{p} \big[ (-1)^{h - 1}
  e_{{i_1},\alpha_1} e_{\alpha_1, {j_1}}e_{{i_2},\alpha_2}   \cdots e_{{i_{h-1}},\alpha_{h-1}}
 e_{{i_{h}},\alpha_h} \cdot e_{\alpha_2, {j_2}}   \cdots e_{\alpha_{h-1}, {j_{h-1}}} e_{\alpha_h, {j_{h}}}
\big] =
$$
$$
= (-1)^{h - 1} \ e_{i_1, j_1} \
\left[
\begin{array}{c}
i_2 \\ \vdots \\ i_{h-1} \\ i_h
\end{array}
\right| \left.
\begin{array}{c}
j_2 \\ \vdots\\ j_{h-1} \\ j_h
\end{array}
\right].
$$

By setting $k = h -i$, we proved the first expansion identity.
The second expansion identity can be proved in a similar way.
\end{proof}

\begin{example}

$$
\left[
\begin{array}{c}
1\\  2  \\ 3
\end{array}
\right| \left.
\begin{array}{c}
2 \\ 1  \\ 1
\end{array}
\right] = [1|2] \left[
\begin{array}{c}
2  \\ 3
\end{array}
\right| \left.
\begin{array}{c}
 1  \\ 1
\end{array}
\right] - \left[
\begin{array}{c}
1  \\ 3
\end{array}
\right| \left.
\begin{array}{c}
 1  \\ 1
\end{array}
\right] = - e_{12}e_{21}e_{31} + e_{11}e_{31} \in \mathbf{U}(gl(n)).
$$
Notice that
$$
\left[
\begin{array}{c}
1\\  2  \\ 3
\end{array}
\right| \left.
\begin{array}{c}
2 \\ 1  \\ 1
\end{array}
\right] = \left[
\begin{array}{c}
3\\  2  \\ 1
\end{array}
\right| \left.
\begin{array}{c}
1 \\ 1  \\ 2
\end{array}
\right] = [3|1]  \left[
\begin{array}{c}
  2  \\ 1
\end{array}
\right| \left.
\begin{array}{c}
 1  \\ 2
\end{array}
\right] - \left[
\begin{array}{c}
  3  \\ 2
\end{array}
\right| \left.
\begin{array}{c}
 2  \\ 1
\end{array}
\right] =
$$
$$
= - [3|1] ( [2|1][[1|2] - [2|2] ) + [2|1][3|2]  = -
e_{31}e_{21}e_{12} + e_{31}e_{22} + e_{21}e_{32} =
$$
$$
= \left[
\begin{array}{c}
 1  \\ 2 \\ 3
\end{array}
\right| \left.
\begin{array}{c}
 2  \\ 1 \\ 1
\end{array}
\right] = \left[
\begin{array}{c}
 1  \\ 2
\end{array}
\right| \left.
\begin{array}{c}
 2  \\ 1
\end{array}
\right]  [3|1] =
$$
$$
= (- [1|2][2|1] +[1|1]) [3|1] = - e_{12}e_{21}e_{31} + e_{11}e_{31}
\in \mathbf{U}(gl(n)).
$$

\end{example}\qed

\begin{remark}
Theorems \emph{\ref{KBT corr}} and \emph{\ref{image symm}},
in combination with Proposition \emph{\ref{devirtualization column}},
allows the \emph{explicit devirtualized} forms in $\mathbf{U}(gl(n))$ of Capelli bitableaux and of right
Young-Capelli bitableaux to be easily computed. The process can be illustrated by an example. Let $n \geq 2,
\ h = 3, \ \lambda = (2, 1).$ Consider the Capelli bitableaux
$$
\left[
\begin{array}{cc}
 1 & 2 \\  1
\end{array}
\right| \left.
\begin{array}{cc}
1 & 2 \\ 2 &
\end{array}
 \
\right] \in \mathbf{U}(gl(n)).
$$
By Theorem \emph{\ref{KBT corr}}:
\begin{align*}
\left[
\begin{array}{cc}
 1 & 2 \\  1
\end{array}
\right| \left.
\begin{array}{cc}
1 & 2 \\ 2 &
\end{array}
 \
\right]
&=
\mathcal{K}^{-1} \big(
\left(
\begin{array}{cc}
 1 & 2 \\  1
\end{array}
\right| \left.
\begin{array}{cc}
1 & 2 \\ 2 &
\end{array}
 \
\right)
\big)
\\
&=
\mathcal{K}^{-1} \big(
\left(
\begin{array}{c}
 1  \\ 2 \\ 1
\end{array}
\right| \left.
\begin{array}{c}
 1  \\ 2 \\ 2
\end{array}
\right)
-
\left(
\begin{array}{c}
 1  \\ 2 \\ 1
\end{array}
\right| \left.
\begin{array}{c}
 2  \\ 1 \\ 2
\end{array}
\right)
\big)
\\
&=
\left[
\begin{array}{c}
 1  \\ 2 \\ 1
\end{array}
\right| \left.
\begin{array}{c}
 1  \\ 2 \\ 2
\end{array}
\right]
-
\left[
\begin{array}{c}
 1  \\ 2 \\ 1
\end{array}
\right| \left.
\begin{array}{c}
 2  \\ 1 \\ 2
\end{array}
\right].
\end{align*}
By Proposition \emph{\ref{devirtualization column}},
\begin{align*}
&\left[
\begin{array}{c}
 1  \\ 2 \\ 1
\end{array}
\right| \left.
\begin{array}{c}
 1  \\ 2 \\ 2
\end{array}
\right]
=
- e_{11}e_{22}e_{12} + e_{12}e_{21} - e_{12} \in \mathbf{U}(gl(n)),
\\
&\left[
\begin{array}{c}
 1  \\ 2 \\ 1
\end{array}
\right| \left.
\begin{array}{c}
 2  \\ 1 \\ 2
\end{array}
\right]
=
- e_{12}e_{21}e_{12} + e_{12}e_{22} + e_{11}e_{12} - e_{12} \in \mathbf{U}(gl(n)).
\end{align*}
\end{remark}\qed

\subsubsection{Column Capelli bitableaux as polynomial differential operators on ${\mathbb C}[M_{n,d}]$}

The next result will play a crucial role in Section \ref{sec quantum}.
In the language of Procesi (\cite{Procesi-BR}, chapter $3$), it describes the action of
column Capelli bitableaux as elements of the {\textit{Weyl algebra}} associated to
the polynomial algebra ${\mathbb C}[M_{n,d}]$.

\begin{proposition}\label{column differential}
The action of the column Capelli bitableau
$$
\left[
\begin{array}{c}
i_1\\  i_2 \\ \vdots \\ i_h
\end{array}
\right| \left.
\begin{array}{c}
j_1\\  j_2 \\ \vdots \\ j_h
\end{array}
\right] \in \mathbf{U}(gl(n))
$$
on the algebra ${\mathbb C}[M_{n,d}]$ equals the action of the polynomial differential operator
$$
(-1)^{h \choose 2}
\sum_{(\varphi_1, \varphi_2, \ldots, \varphi_h) \in \underline{d}^h} \
(i_1|\varphi_1) (i_2|\varphi_2) \cdots (i_h|\varphi_h) \
\partial_{(j_1|\varphi_1)}
\ \partial_{(j_2|\varphi_2)} \cdots \partial_{(j_h|\varphi_h)},
$$
\end{proposition}

\begin{proof}
Consider a monomial $\mathbb{M} \in {\mathbb C}[M_{n,d}]$,
$$
\mathbb{M} = \prod_{i = 1}^n \ (i|1)^{s_{i1}}(i|2)^{s_{i2}} \cdots (i|d)^{s_{id}}
$$
and let $\alpha \in A_0$ be a {\textit{positive virtual symbol}}.
Given $j_h = 1, 2, \ldots, n$, consider the action of the superpolarization
$D_{\alpha, j_h}$ on the supersymmetric algebra  $ \mathbb{C}[M_{m_0|m_1+n,d}] \supseteq \mathbb C[M_{n,d}]$.
A straightforward computation shows that
\begin{equation}\label{differential super}
D_{\alpha, j_h}\big( \mathbb{M} \big) = \sum_{\varphi = 1}^d \ \partial_{(j_1|\varphi)}\big( \mathbb{M} \big)(\alpha|\varphi).
\end{equation}
Furthermore, notice that:
$$
D_{{\alpha_s}, {j_k}}D_{{\alpha_t}, {j_h}}\big( \mathbb{M} \big) = \sum_{\varphi = 1}^d \ \Big( D_{\alpha_s, j_k}
\big( \partial_{(j_1|\varphi)} \big( \mathbb{M} \big) \big) \Big) \ (\alpha_t|\varphi),
$$
that equals
\begin{equation}\label{commmutation}
\sum_{\varphi_1, \varphi_2 = 1, 2, \ldots, d} \
\Big( \partial_{(j_k|\varphi_1)} \partial_{(j_h|\varphi_2)} \big( \mathbb{M} \big) \Big) \ (\alpha_s|\varphi_1)(\alpha_t|\varphi_2).
\end{equation}

Recall that the action of the column Capelli bitableau
$$
\left[
\begin{array}{c}
i_1\\  i_2 \\ \vdots \\ i_h
\end{array}
\right| \left.
\begin{array}{c}
j_1\\  j_2 \\ \vdots \\ j_h
\end{array}
\right] \in \mathbf{U}(gl(n))
$$
on the algebra ${\mathbb C}[M_{n,d}]$ is implemented by the product of superpolarizations
$$
D_{i_1, \alpha_1} \cdots D_{i_{h-1}, \alpha_{h-1}}  D_{i_h, \alpha_h} D_{\alpha_1, j_1} \cdots D_{\alpha_{h-1} , j_{h-1}} D_{\alpha_h, j_h},
$$
where $\alpha_1, \ldots, \alpha_{h-1}, \alpha_h$ are distinct arbitrary positive virtual symbols.
Note that $| D_{i_r, \alpha_r}| = |D_{\alpha_r, j_r}| = 1 \in \mathbb{Z}_2$, for every $r = 1,2, \ldots, h.$

From Eqs. (\ref{differential super}) and (\ref{commmutation}), it immediately follows:
\begin{equation}\label{virtual part}
D_{\alpha_1, j_1} \cdots  D_{\alpha_h, j_h}\big( \mathbb{M} \big) =
\sum_{(\varphi_1, \ldots, \varphi_h) \in \underline{d}^h}
 \partial_{(j_1|\varphi_1)} \cdots \partial_{(j_h|\varphi_h)}\big( \mathbb{M} \big)
 (\alpha_1|\varphi_1)  \cdots (\alpha_h|\varphi_h)
\end{equation}

Since $|(\alpha_r|\varphi_r)| = 1 \in \mathbb{Z}_2$, for every $r = 1,2, \ldots, h,$ from
Eq. (\ref{virtual part}), we infer:
$$
D_{i_1, \alpha_1} \cdots D_{i_{h-1}, \alpha_{h-1}}  D_{i_h, \alpha_h}
\Big( D_{\alpha_1, j_1} \cdots D_{\alpha_{h-1} , j_{h-1}} D_{\alpha_h, j_h}\big( \mathbb{M} \big) \Big)
$$
equals
$$
(-1)^{h \choose 2}
\sum_{(\varphi_1, \varphi_2, \ldots, \varphi_h) \in \underline{d}^h} \
(i_1|\varphi_1) (i_2|\varphi_2) \cdots (i_h|\varphi_h) \ \partial_{(j_1|\varphi_1)} \ \partial_{(j_2|\varphi_2)}
\cdots \partial_{(j_h|\varphi_h)}\big( \mathbb{M} \big).
$$
\end {proof}

\section{Capelli immanants and Young-Capelli bitableaux in $\mathbf{U}(gl(n))$}\label{Capimm sect}}

The \emph{bitableaux correspondence} (linear)  isomorphism
$$
\mathcal{K}^{-1} : {\mathbb C}[M_{n,n}] \rightarrow \mathbf{U}(gl(n)),
$$
leads to the following \emph{natural} definition of  {\textit{Capelli
immanant}}
$$
Cimm_{\lambda}[i_1 i_2 \cdots i_h;j_1 j_2 \cdots j_h]
$$
in  the enveloping algebra in $\mathbf{U}(gl(n))$:
$$
Cimm_{\lambda}[i_1 i_2 \cdots i_h;j_1 j_2 \cdots j_h] = \mathcal{K}^{-1}  \Big( imm_{\lambda}(i_1 i_2 \cdots i_h;j_1 j_2 \cdots j_h) \Big).
$$

By linearity of the operator $\mathcal{K}^{-1}$, we get
\begin{align*}
Cimm_{\lambda}[i_1 i_2 \cdots i_h;j_1 j_2 \cdots j_h] & =
\sum_{\sigma \in \mathbf{S}_h} \ \chi^{\lambda}(\sigma) \left[
\begin{array}{c}
i_{\sigma(1)}\\  i_{\sigma(2)} \\ \vdots \\ i_{\sigma(h)}
\end{array}
\right| \left.
\begin{array}{c}
j_1\\  j_2 \\ \vdots \\ j_h
\end{array}
\right]
\\
& =\sum_{\sigma \in \mathbf{S}_h} \ \chi^{\lambda}(\sigma)  \left[
\begin{array}{c}
i_1\\  i_2 \\ \vdots \\ i_h
\end{array}
\right| \left.
\begin{array}{c}
j_{\sigma(1)}\\  j_{\sigma(2)} \\ \vdots \\ j_{\sigma(h)}
\end{array}
\right].
\end{align*}

Clearly, the notion of Capelli immanants provides a natural generalization of the notion of
Capelli determinant  (see Example \ref{Capelli determinants-BR}).

Since a Young-Capelli bitableau $[U|\fbox{$V$}] \in \mathbf{U}(gl(n))$ is the
image of the right symmetrized bitableau $(U|\fbox{$V$}) \in {\mathbb C}[M_{n,n}]$
with respect to the isomorphism $\mathcal{K}^{-1}$,
Proposition \ref{imm vs bit} implies

\begin{theorem}\label{Cimm vs YC}
Let $\lambda \vdash h$. Any Capelli immanant $Cimm_{\lambda}[i_1 i_2 \cdots i_h;j_1 j_2 \cdots j_h] $
can be written as  a linear combination of standard Young-Capelli bitableaux $[U| \fbox{$V$}]$
in $\mathbf{U}(gl(n))$ of the \textbf{same} shape $\lambda$:
\begin{multline*}
Cimm_{\lambda}[i_1 i_2 \cdots i_h;j_1 j_2 \cdots j_h] =
 \sum_{U, V} \ \varrho_{U,V}  \ [U|\fbox{$V$}], \\ \varrho_{U,V} \in {\mathbb C}, \quad sh(U) = sh(V) = \lambda.
\end{multline*}
\end{theorem}

From Corollary \ref{zero imm}, it follows:

\begin{corollary}\label{zero Cimm}
Let $\lambda \vdash h$. If $\lambda_1  \nleq n$, then
$$
Cimm_{\lambda}(i_1 i_2 \cdots i_h;j_1 j_2 \cdots j_h) = 0.
$$
\end{corollary}

Furthermore,  Proposition \ref{bit vs imm} implies
\begin{theorem}\label{Cbit vs Cimm}
Let $\lambda \vdash h$. Any Young-Capelli bitableau $[U| \fbox{$V$}]$  in $\mathbf{U}(gl(n))$ of shape $sh(U) = sh(V) = \lambda$
can be written as  a linear combination of Capelli immanants $Cimm_{\lambda}[i_1 i_2 \cdots i_h;j_1 j_2 \cdots j_h]$
associated to the \textbf{same} shape $\lambda$.
\end{theorem}

Proposition \ref{immanant span} implies

\begin{theorem}\label{Capelli immanant span}
The set of Capelli immanants
$$
 \bigcup_{k = 0}^h  \  \big\{ Cimm_{\lambda}[i_1 i_2 \cdots i_k;j_1 j_2 \cdots j_k]; \lambda \vdash k, \lambda_1 \leq n,
 \
(i_1 i_2 \cdots i_k), (j_1 j_2 \cdots j_k) \in \underline{n}^k \big\}
$$
is a spanning set of $\mathbf{U}(gl(n))^{(h)}.$
\end{theorem}

\section{Quantum immanants}\label{sec quantum}

Our main result is a description of {\textit{quantum immanants}} as simple linear combinations
of {\textit{Capelli immanants}}.  This result -- in combination with Proposition \ref{devirtualization column} --
allows the computation of quantum immanants as elements of $\mathbf{U}(gl(n))$ to be reduced to a fairly simple process
(see, e.g. Example \ref{Schur example} below).

Quantum immanants  are  the preimages
of the shifted Schur polynomials  \cite{Sahi1-BR},  \cite{OkOlsh-BR},
with respect to the Harish-Chandra isomorphism.

We  follow the notational conventions of Okounkov \cite{Okounkov-BR} and \cite{Okounkov1-BR}.

\begin{remark}
Given a partition $\mu \vdash h$, $V^{\mu}$ denotes the irreducible representation associated to $\mu$
in the sense of James and Kerber \cite{JK}. We  recall that the  irreducible representation $V^{\mu}$ is the representation associated to the shape $\widetilde{\mu}$ in the notation of the previous sections of this work.
\end{remark}

Furthermore:

\begin{itemize}

\item[--] $\mathbf{T}$ denotes a \emph{multilinear standard} Young tableau of shape $sh(\mathbf{T}) = \mu \vdash h$.

\item[--] For every $s = 1, 2, \ldots, h,$ let $(i, j)$ be the pair of row and column indices of the cell
of $\mathbf{T}$ that contains $s$. Set  $\mathbf{c}_\mathbf{T}(s) = j - i$ (the ``Frobenius content'' of the cell $(i, j)$).

\item[--] $v_\mathbf{T}$ denotes the element of the {\textit{seminormal Young basis}} of
$V^{\mu}$ associated to the multilinear standard tableau $\mathbf{T}$.
Since each basis vector is defined only up to a scalar factor, we assume that $(v_\mathbf{T}, v_\mathbf{T}) = 1$
(see   Okounkov and Vershik \cite{OkVer-BR}; for a more traditional approach, see James and Kerber \cite{JK}).

\item[--] given the element
$$
\Psi_\mathbf{T} = \sum_{\sigma \in  \mathbf{S}_h} \ (\sigma \cdot v_\mathbf{T}, v_\mathbf{T})\sigma^{-1}
\in \mathbb{C}[\mathbf{S}_h],
$$
$\Psi_\mathbf{T}^h$ denotes the matrix that represents the element
 $\Psi_\mathbf{T}$ as a linear operator on the tensor space $(\mathbb{C}^n)^{\otimes h}$.

\item[--] Let $E = [e_{ij}]_{i, j = 1,2 , \ldots ,n}$ be the matrix whose entries are the elements
of the standard basis of $gl(n)$.

\item[--] Let
$$
\mathbb{E}_\mathbf{T} =
 \big( (E - \mathbf{c}_\mathbf{T}(1)) \otimes (E - \mathbf{c}_\mathbf{T}(2)) \otimes \cdots \otimes
 (E - \mathbf{c}_\mathbf{T}(h)) \big) \  \Psi_\mathbf{T}^h
$$
be the {\textit{fusion}} matrix; the fusion matrix  $\mathbb{E}_\mathbf{T}$ is a
$(n^h \times n^h)-$matrix with entries in $\mathbf{U}(gl(n)).$
\end{itemize}

 Following Okounkov (\cite{Okounkov1-BR}, \cite{Okounkov-BR}), the element
$$
Tr \big( \mathbb{E}_\mathbf{T} \big) \in \mathbf{U}(gl(n))^{(h)}
$$
is  the {\textit{quantum immanant}} associated to the multilinear standard tableau $\mathbf{T}$, $sh(\mathbf{T}) = \mu$.

The {\textit{higher Capelli identities}} (\cite{Okounkov1-BR}, \cite{Okounkov-BR}),
imply (\cite{Okounkov1-BR}, Eq. (5.1)) that the action
of the quantum immanant
$$
Tr \big( \mathbb{E}_\mathbf{T} \big), \quad sh(\mathbf{T}) = \mu
$$
on the algebra ${\mathbb C}[M_{n,d}]$ equals the action of the polynomial differential operator
\begin{equation}\label{higher Capelli}
\frac {1} {dim(V^{\mu})} \ Tr \Big( X^{\otimes h} \ (D')^{\otimes h} \ \overline{\chi}_{\widetilde{\mu}}^h \Big),
\end{equation}
where
\begin{itemize}

\item[--] $X$ denotes the matrix $\big[ (i|\varphi) \big]_{i = 1,  \ldots, n; \varphi = 1, \ldots, d} $

\item[--] $D$ denotes the matrix $\big[ \partial_{(i|\varphi)} \big]_{i = 1,  \ldots, n; \varphi = 1, \ldots, d} $ of
partial derivatives on the algebra ${\mathbb C}[M_{n,d}]$,
and the prime stands for transposition.

\item[--] $\overline{\chi}_{\widetilde{\mu}}^h$ denotes the matrix that represents the element
$$
\overline{\chi}_{\widetilde{\mu}} = \sum_{\sigma \in \mathbf{S}_h} \ \chi^{\widetilde{\mu}}(\sigma) \sigma \in {\mathbb C}[\mathbf{S}_h]
$$
of Eq. (\ref{character})  as a linear operator on the tensor space $(\mathbb{C}^n)^{\otimes h}$.

\end{itemize}

Since the action of $\mathbf{U}(gl(n))$
on the algebra ${\mathbb C}[M_{n,d}]$ is a faithful action whenever $n \leq d$,
and the differential operator of Eq. (\ref{higher Capelli}) is independent from the choice
of the multilinear standard tableau $\mathbf{T}$, the quantum immanant $Tr \big( \mathbb{E}_\mathbf{T} \big) $ only depends
on the shape $\mu$.

\begin{theorem}\label{quantum}
The quantum immanant
$$
Tr \big( \mathbb{E}_\mathbf{T} \big), \quad sh(\mathbf{T}) = \mu
$$
equals the linear combination of Capelli immanants:
\begin{equation}\label{quantum first}
(-1)^{h \choose 2}
\sum_{h_1 + h_2 + \cdots + h_n = h} \ \frac {H(\mu)} {h_1! h_2! \cdots h_n!}
 \ Cimm_{\widetilde{\mu}}[1^{h_1}2^{h_2} \ldots n^{h_n}; 1^{h_1}2^{h_2} \ldots n^{h_n}],
\end{equation}
where $1^{h_1}2^{h_2} \ldots n^{h_n}$ is a short notation for the non decreasing sequence
$i_1 i_2 \cdots i_h$ with
$$
h_p = \sharp \{i_q = p;\ q = 1, 2, \ldots, h \}, \quad p = 1, 2, \ldots, n.
$$
\end{theorem}

\begin{proof}

For every $\sigma \in \mathbf{S}_h$, $\overline{i} = (i_1, \ldots, i_h) \in \underline{n}^h$,
$\overline{\varphi} = (\varphi_1, \ldots, \varphi_h) \in \underline{d}^h$, we set
$$
P_{\sigma}[\overline{i};\overline{\varphi}] = (i_1|\varphi_1)  \cdots (i_h|\varphi_h) \
\partial_{(i_{\sigma(1)}|\varphi_1)} \cdots \partial_{(i_{\sigma(h)}|\varphi_h)}.
$$

By straightforward computation, the right-hand side of Eq. (\ref{higher Capelli}) equals
\begin{equation}
\frac {1} {dim(V^{\mu})} \ \sum_{\overline{i} = (i_1, \ldots, i_h) \in \underline{n}^h} \
\Big( \sum_{\sigma \in \mathbf{S}_h} \ \chi^{\widetilde{\mu}}(\sigma) \
\big( \sum_{\overline{\varphi} = (\varphi_1, \ldots, \varphi_h) \in \underline{d}^h} \
P_{\sigma}[\overline{i};\overline{\varphi}] \big) \Big).
\end{equation}

By Proposition \ref{column differential}, the action of the Capelli immanant
$$
Cimm_{\widetilde{\mu}}[i_1 i_2 \cdots i_h;i_1 i_2 \cdots i_h] \in \mathbf{U}(gl(n))
$$
on the algebra ${\mathbb C}[M_{n,d}]$ equals the action of the polynomial differential operator
$$
(-1)^{h \choose 2} \ \sum_{\sigma \in \mathbf{S}_h} \ \chi^{\widetilde{\mu}}(\sigma) \
\big( \sum_{\overline{\varphi} = (\varphi_1, \ldots, \varphi_h) \in \underline{d}^h} \
P_{\sigma}[\overline{i};\overline{\varphi}] \big),
$$
for every $\overline{i} = (i_1, \ldots, i_h) \in \underline{n}^h$.

Since the action of $\mathbf{U}(gl(n))$
on the algebra ${\mathbb C}[M_{n,d}]$ is a faithful action whenever $n \leq d$,
it immediately follows that any {\textit{quantum immanant}} equals - up to a scalar factor - the following
linear combinatio of {\textit{Capelli immanants}}:
\begin{equation}\label{quantum Capelli}
Tr \big( \mathbb{E}_\mathbf{T} \big) = (-1)^{h \choose 2} \frac {1} {dim(V^{\mu})} \
\sum_{(i_1, \ldots, i_h) \in \underline{n}^h} \ Cimm_{\widetilde{\mu}}[i_1 i_2 \cdots i_h;i_1 i_2 \cdots i_h] \in \mathbf{U}(gl(n)).
\end{equation}
Since
$$
Cimm_{\widetilde{\mu}}[i_1 i_2 \cdots i_h;i_1 i_2 \cdots i_h] =
Cimm_{\widetilde{\mu}}[i_{\tau(1)} i_{\tau(2)} \cdots i_{\tau(h)};i_{\tau(1)} i_{\tau(2)} \cdots i_{\tau(h)}],
$$
for every $\tau \in \mathbb{C}[\mathbf{S}_h]$, the right-hand side of Eq. (\ref{quantum Capelli})
equals
$$
(-1)^{h \choose 2}
\sum_{h_1 + h_2 + \cdots + h_n = h} \ \frac {H(\mu)} {h_1! h_2! \cdots h_n!}
 \ Cimm_{\widetilde{\mu}}[1^{h_1}2^{h_2} \ldots n^{h_n}; 1^{h_1}2^{h_2} \ldots n^{h_n}].
$$
\end{proof}

From Theorem \ref{quantum} and Corollary \ref{zero Cimm}, it follows:

\begin{corollary}
Let $\mathbf{T}$ be a multilinear standard tableau, $sh(\mathbf{T}) = \mu$.

If $\widetilde{\mu}_1  \nleq n$, then
$$
Tr \big( \mathbb{E}_\mathbf{T} \big) = 0.
$$
\end{corollary}\qed

Let $\mu$ with $\widetilde{\mu}_1 \leq n$, and let recall that $\boldsymbol{\zeta}(n)$ is the center of $\mathbf{U}(gl(n))$.
According with Okoukov \cite{Okounkov1-BR}, \cite{Okounkov-BR},
the {\textit{Schur element}} $\mathbf{S}_\mu(n) \in \boldsymbol{\zeta}(n)$   is defined by setting
$$
\mathbf{S}_\mu(n) = \frac {dim(V^{\mu})} {h!} \ Tr \big( \mathbb{E}_\mathbf{T} \big).
$$
Since
$
dim(V^{\mu}) = \frac {h!} {H(\mu)},
$
Theorem \ref{quantum} implies:

\begin{corollary}\label{schur first}
\begin{equation}\label{schur first eq}
\mathbf{S}_\mu(n) = (-1)^{h \choose 2}
\sum_{h_1 + \cdots + h_n = h} \ \frac {1} {h_1!  \cdots h_n!}
 \ Cimm_{\widetilde{\mu}}[1^{h_1} \ldots n^{h_n}; 1^{h_1} \ldots n^{h_n}].
\end{equation}
\end{corollary}

\begin{remark}
If $\mu = (1^h)$, $h \leq n$, is the column shape of length $h$, then $\mathbf{S}_{(1^h)}(n)$  is immediately recognized as  the
$h-$th \emph{determinantal Capelli generator} $\mathbf{H}_h$  (see, e.g. \emph{\cite{Brini4-BR}, \cite{BriUMI-BR}};
see also
Capelli \emph{\cite{Cap1-BR}, \cite{Cap2-BR}, \cite{Cap3-BR}} and \emph{\cite{Cap4-BR}}, Howe and Umeda
\emph{\cite{HU-BR}}).

If $\mu = (h)$ is the row shape of length $h$, then $\mathbf{S}_{(h)}(n)$  is immediately recognized as  the
$h-$th \emph{permanental Nazarov--Umeda generator} $\mathbf{I}_h$ (see, e.g. \emph{\cite{BriUMI-BR}},  Nazarov \emph{\cite{Nazarov-BR}}, Umeda \emph{\cite{UmedaHirai-BR}}).
\end{remark}

\begin{example}\label{quantum Schur}
Let $h = 3$, $\mu = (2, 1) = \widetilde{\mu}$, $n = 2$. Recall that ${H(\mu)} = 3$. Then

\begin{align}
\mathbf{S}_{(2, 1)}(2) &= - \frac {1} {2} \ \Big( Cimm_{(2, 1)}[112; 112] + Cimm_{(2, 1)}[122; 122] \Big)
\\
=&
- \left[
\begin{array}{c}
1\\  1 \\  2
\end{array}
\right| \left.
\begin{array}{c}
1\\  1 \\  2
\end{array}
\right]
+
\left[
\begin{array}{c}
1\\  1 \\  2
\end{array}
\right| \left.
\begin{array}{c}
1\\  2 \\  1
\end{array}
\right]
- \left[
\begin{array}{c}
1\\  2 \\  2
\end{array}
\right| \left.
\begin{array}{c}
1\\  2 \\  2
\end{array}
\right]
+
\left[
\begin{array}{c}
1\\  2 \\  2
\end{array}
\right| \left.
\begin{array}{c}
2\\  1 \\  2
\end{array}
\right]. \label{Schur example}
\end{align}

By Eq. (\ref{Schur example})  and Proposition \ref{devirtualization column}, we have:
\begin{align*}
\mathbf{S}_{(2, 1)}(2) =&
+ e_{1 1}^2e_{2 2} - e_{1 1}e_{2 2} - e_{1 1}e_{1 2}e_{2 1} + e_{1 1}^2 + e_{1 2}e_{2 1} - e_{1 1}
\\
&
+ e_{1 1}e_{2 2}^2 - e_{1 1}e_{2 2}
- e_{1 2}e_{2 1}e_{2 2} + e_{1 1}e_{2 2} + e_{1 2}e_{2 1} - e_{1 1}
\\
=&
+ e_{1 1}^2e_{2 2} - e_{1 1}e_{1 2}e_{2 1} + e_{1 1}e_{2 2}^2 - e_{1 2}e_{2 1}e_{2 2}
\\
&
- e_{1 1}e_{2 2}  + e_{1 1}^2 + 2 e_{1 2}e_{2 1} - 2 e_{1 1} \in \mathbf{U}(gl(2)).
\end{align*} \qed
\end{example}

\begin{remark}\label{quantum Schur}
According to \emph{Theorem \ref{Cimm vs YC}}, the central element  $\mathbf{S}_{(2, 1)}(2)$ also equals,
in turn, a linear combination of Young-Capelli bitableaux. Indeed, we have
\begin{align*}\label{Schur el}
&- \frac {1} {2} \
\left[
\begin{array}{cc}
 1 & 2 \\  1
\end{array}
\right| \left.
\fbox{$
\begin{array}{cc}
1 & 2 \\ 1 &
\end{array}
$} \
\right]
-
\left[
\begin{array}{cc}
 1 & 2 \\  2
\end{array}
\right| \left.
\fbox{$
\begin{array}{cc}
1 & 2 \\ 2 &
\end{array}
$} \
\right]
\\
=& - \left[
\begin{array}{c}
1\\  1 \\  2
\end{array}
\right| \left.
\begin{array}{c}
1\\  1 \\  2
\end{array}
\right]
+
\left[
\begin{array}{c}
1\\  1 \\  2
\end{array}
\right| \left.
\begin{array}{c}
1\\  2 \\  1
\end{array}
\right]
- \left[
\begin{array}{c}
1\\  2 \\  2
\end{array}
\right| \left.
\begin{array}{c}
1\\  2 \\  2
\end{array}
\right]
+
\left[
\begin{array}{c}
1\\  2 \\  2
\end{array}
\right| \left.
\begin{array}{c}
2\\  1 \\  2
\end{array}
\right] =
 \
\mathbf{S}_{(2, 1)}(2).
\end{align*}

The previous identity  is an instance of an alternative presentation
(see our preliminary manuscript \emph{\cite{BriniTeolis-BR}}, \emph{Subsection $4.5.1$})
of the Schur element $\mathbf{S}_\mu(n) \in \boldsymbol{\zeta}(n), \ \widetilde{\mu}_1 \leq n$
of \emph{Corollary \ref{schur first}}. \qed
\end{remark}

Given a pair of row (strictly) increasing tableaux $S$ and $T$  of shape
$sh(S) = sh(T) = \widetilde{\mu} \vdash   h$ on the proper alphabet $L = \{1, 2,  \ldots, n \}$, consider the element
$$
e_{S,C^*_{\widetilde{\mu}}} \cdot e_{C^*_{\widetilde{\mu}},D^*_{\widetilde{\mu}}}
 \cdot  e_{D^*_{\widetilde{\mu}},C^*_{\widetilde{\mu}}}
\cdot e_{C^*_{\widetilde{\mu}},T} \in Virt(m_0+m_1,n) \subset {\mathbf{U}}(gl(m_0|m_1+n)).
$$
We set
\begin{equation}\label{rightYoungCapelli}
[\ \fbox{$S \ | \ T$}\ ]  = \mathfrak{p} \Big( e_{S,C^*_{\widetilde{\mu}}} \cdot e_{C^*_{\widetilde{\mu}},D^*_{\widetilde{\mu}}}
 \cdot  e_{D^*_{\widetilde{\mu}},C^*_{\widetilde{\mu}}}
\cdot e_{C^*_{\widetilde{\mu}},T} \Big)
          \in {\mathbf{U}}(gl(n)).
\end{equation}
and call the element $[\ \fbox{$S \ | \ T$}\ ] $ a {\textit{double Young-Capelli bitableau}}.

\begin{proposition} Any \emph{double Young-Capelli bitableau}
equals a sum of Young-Capelli bitableaux:
$$
[\ \fbox{$S \ | \ T$}\ ] = (-1)^ {h \choose 2} \ \sum_{\sigma} \ (-1)^{|\sigma|} \  [S|\fbox{$T^{\sigma}$}],
$$
where the sum is extended to all  Young tableaux $T^{\sigma}$ obtained from $T$ by  permutations of the elements
of each row, and $(-1)^{|\sigma|}$ is the product of the signatures of row permutations.
\end{proposition}
\begin{proof} See our preliminary manuscript \cite{BriniTeolis-BR}.
\end{proof}

\begin{theorem}\label{quantum second}
We have

\begin{equation}\label{schur second eq}
\mathbf{S}_\mu(n) =
\frac {1} {H(\mu)} \  \ \sum_S \ [\ \fbox{$S \ | \ S$}\ ]  \in {\mathbf{U}}(gl(n))
\end{equation}
where the sum is extended to all  row (strictly) increasing tableaux $S$ of shape $sh(S) = \widetilde{\mu} \vdash h$
on the proper alphabet $L = \{1, 2,  \ldots, n \}$.
\end{theorem}
\begin{proof}(Sketch) Let $\chi_n$ be the Harish-Chandra isomorphism
$$
\chi_n : \boldsymbol{\zeta}(n) \longrightarrow \Lambda^*(n),
$$
where  $\boldsymbol{\zeta}(n)$ is the center $\boldsymbol{\zeta}(n)$  of ${\mathbf{U}}(gl(n))$, and
$\Lambda^*(n)$ is the algebra of
{\it{shifted symmetric polynomials}} in $n$ variables (see, e.g. \cite{OkOlsh-BR}).

The technique is to show that both sides of Eq. (\ref{schur second eq}) have the same image under the
isomorphism $\chi_n$.

The right-hand side of Eq. (\ref{schur second eq}) is easily proved to be an
element of the center $\boldsymbol{\zeta}(n)$, and its image via the Harish-Chandra isomorphism satisfies
the hypotheses (see Theorem $6.64$ of \cite{BriniTeolis-BR}) of the {\it{Sahi/Okounkov Characterization Theorem}}
(Theorem 1 of \cite{Sahi1-BR} and Theorem 3.3 of \cite{OkOlsh-BR}, see also \cite{Okounkov-BR})
for the {\it{Schur shifted symmetric polynomial}} $s_{\mu|n}^*$ of \cite{OkOlsh-BR}. Then
$$
\chi_n \left( \frac {1} {H(\mu)} \  \ \sum_S \ [\ \fbox{$S \ | \ S$}\ ] \right) = s_{\mu|n}^*.
$$
Since
$$
\chi_n \left( \frac {dim(V^{\mu})} {h!} \ Tr \big( \mathbb{E}_\mathbf{T} \big) \right) = s_{\mu|n}^*
$$
(see \cite{Okounkov1-BR}, \cite{Okounkov-BR}), the assertion follows
(for details, see our preliminary manuscript \cite{BriniTeolis-BR}).
\end{proof}

\begin{example} We have
\begin{align*}
\mathbf{S}_{(2, 1)}(2) =& \ \frac {1} {3} \
\Big(
\left[ \
\fbox{$
\begin{array}{cc}
 1 & 2 \\  1
\end{array}
\bigg| \
\begin{array}{cc}
1 & 2 \\ 1 &
\end{array}
$} \
\right]
+
\left[ \
\fbox{$
\begin{array}{cc}
 1 & 2 \\  2
\end{array}
\bigg| \
\begin{array}{cc}
1 & 2 \\ 2 &
\end{array}
$} \
\right]
\Big)
\\
=& \ \frac {1} {3} \ \Big(
- \left[
\begin{array}{cc}
 1 & 2 \\  1
\end{array}
\right| \left.
\fbox{$
\begin{array}{cc}
1 & 2 \\ 1 &
\end{array}
$} \
\right]
+
\left[
\begin{array}{cc}
 1 & 2 \\  1
\end{array}
\right| \left.
\fbox{$
\begin{array}{cc}
2 & 1 \\ 1 &
\end{array}
$} \
\right]
\\
&
\phantom{\ \frac {1} {3} \ \big(}
- \left[
\begin{array}{cc}
 1 & 2 \\  2
\end{array}
\right| \left.
\fbox{$
\begin{array}{cc}
1 & 2 \\ 2 &
\end{array}
$} \
\right]
+
\left[
\begin{array}{cc}
 1 & 2 \\  2
\end{array}
\right| \left.
\fbox{$
\begin{array}{cc}
2 & 1 \\ 2 &
\end{array}
$} \
\right]
\Big).
\end{align*}

Since
$$
- \left[
\begin{array}{cc}
 1 & 2 \\  1
\end{array}
\right| \left.
\fbox{$
\begin{array}{cc}
1 & 2 \\ 1 &
\end{array}
$} \
\right]
+
\left[
\begin{array}{cc}
 1 & 2 \\  1
\end{array}
\right| \left.
\fbox{$
\begin{array}{cc}
2 & 1 \\ 1 &
\end{array}
$} \
\right] =
\ - \frac {3} {2} \
\left[
\begin{array}{cc}
 1 & 2 \\  1
\end{array}
\right| \left.
\fbox{$
\begin{array}{cc}
1 & 2 \\ 1 &
\end{array}
$} \
\right]
$$
and
$$
- \left[
\begin{array}{cc}
 1 & 2 \\  2
\end{array}
\right| \left.
\fbox{$
\begin{array}{cc}
1 & 2 \\ 2 &
\end{array}
$} \
\right]
+
\left[
\begin{array}{cc}
 1 & 2 \\  2
\end{array}
\right| \left.
\fbox{$
\begin{array}{cc}
2 & 1 \\ 2 &
\end{array}
$} \
\right] =
\ - 3 \
\left[
\begin{array}{cc}
 1 & 2 \\  2
\end{array}
\right| \left.
\fbox{$
\begin{array}{cc}
1 & 2 \\ 2 &
\end{array}
$} \
\right],
$$
then
\begin{align*}
\mathbf{S}_{(2, 1)}(2) &= \ \frac {1} {3} \
\Big(
\left[ \
\fbox{$
\begin{array}{cc}
 1 & 2 \\  1
\end{array}
\bigg| \
\begin{array}{cc}
1 & 2 \\ 1 &
\end{array}
$} \
\right]
+
\left[ \
\fbox{$
\begin{array}{cc}
 1 & 2 \\  2
\end{array}
\bigg| \
\begin{array}{cc}
1 & 2 \\ 2 &
\end{array}
$} \
\right]
\Big)
\\
&= \
- \frac {1} {2} \
\left[
\begin{array}{cc}
 1 & 2 \\  1
\end{array}
\right| \left.
\fbox{$
\begin{array}{cc}
1 & 2 \\ 1 &
\end{array}
$} \
\right]
-
\left[
\begin{array}{cc}
 1 & 2 \\  2
\end{array}
\right| \left.
\fbox{$
\begin{array}{cc}
1 & 2 \\ 2 &
\end{array}
$} \
\right],
\end{align*}
as in Remark \ref{quantum Schur}. \qed
\end{example}

Presentation (\ref{schur second eq}) is more supple and effective than presentation (\ref{schur first eq}).
Indeed:
\begin{itemize}

\item [--]
Presentation (\ref{schur second eq}) doesn't involve the irreducible characters of symmetric groups.

\item [--]
Presentation (\ref{schur second eq}) is better suited to the study of the eigenvalues on irreducible $gl(n)-$modules,
and of the duality in the algebra $\boldsymbol{\zeta}(n)$ (see our preliminary manuscript \cite{BriniTeolis-BR}, Section 4).

\item [--]
Presentation (\ref{schur second eq}) is better suited to the study of the limit $n \rightarrow \infty$,
via the \emph{Olshanski decomposition}, see our preliminary manuscript \cite{BriniTeolis-BR}, Section 5,
Olshanski \cite{Olsh1-BR}, \cite{Olsh3-BR} and
Molev \cite{Molev1-BR}, pp. 928 ff.

\end{itemize}


\begin{thebibliography}{99}

%
\bibitem{ABP-BR}
M. Atiyah, R. Bott and V. Patodi, On the heat equation and the Index Theorem, {\it{Invent.
Math.}} {\bf 19}(1973), 279-330

%
\bibitem{BL1-BR}
L. C. Biedenharn and J. D. Louck, A new class of symmetric polynomials
defined in terms of tableaux, {\it Advances in Appl. Math. \/} {\bf 10} (1989), 396--438


%
\bibitem{Bri-BR}
A. Brini, Combinatorics, superalgebras, invariant theory and
representation theory, {\it S\'{e}minaire Lotharingien de Combinatoire} {\bf 55}
(2007), Article B55g, 117  pp.
%
\bibitem{BriUMI-BR}
A. Brini, Superalgebraic Methods in the Classical Theory of Representations. Capelli's Identity, the Koszul
map and the Center of the Enveloping Algebra ${\textbf{U}}(gl(n))$, in {\it Topics in Mathematics, Bologna},
Quaderni dell' Unione Matematica Italiana n. 15, UMI, 2015, pp. 1 -- 27
%
\bibitem{Brini1-BR}
A. Brini, A. Palareti, A. Teolis, Gordan--Capelli series in
superalgebras, {\it Proc. Natl. Acad. Sci. USA\/} {\bf 85} (1988),
1330--1333
%
\bibitem{Brini2-BR}
A. Brini,  A. Teolis, Young--Capelli symmetrizers in
superalgebras, {\it Proc. Natl. Acad. Sci. USA\/} {\bf 86} (1989),
775--778.
%
\bibitem{Brini3-BR}
A. Brini, A. Teolis, Capelli bitableaux and $\mathbb{Z}$-forms of general
linear Lie superalgebras, {\it Proc. Natl. Acad. Sci. USA\/} {\bf
87} (1990), 56--60
%
\bibitem{Brini4-BR}
A. Brini, A. Teolis, Capelli's theory, Koszul maps, and
superalgebras, {\it Proc. Natl. Acad. Sci. USA\/} {\bf 90} (1993),
10245--10249

%
\bibitem{BRT-BR}
A. Brini, F. Regonati, A. Teolis, The method of virtual variables
and Representations of Lie Superalgebras, in {\it Clifford
algebras. Applications to Mathematics, Physics, and
Engineering} (R. Ab\l amowicz, ed.), {\it Progress in
Mathematical Physics}, vol.~34, Birkh\"auser, Boston, 2004,
245--263

%

\bibitem{BriniTeolis-BR}
A. Brini,  A. Teolis, Central elements in $\mathbf{U}(gl(n))$, shifted symmetric functions
and the superalgebraic Capelli's method of virtual variables, preliminary version, Jan. 2018,  arXiv: 1608.06780v4,  73 pp.


%

\bibitem{BriniTeolisKosz-BR}
A. Brini,  A. Teolis, On the action of the Koszul map over the enveloping
algebra of the general linear Lie algebra, preliminary version, March 2020,  arXiv: 1906.02516v2, 23 pp.



%
\bibitem{Cap1-BR}
A. Capelli, Ueber die Zur\"{u}ckf\"{u}hrung der Cayley'schen
Operation $\Omega$ auf gew\"{o}hnliche Polar-Operationen,
\textit{Math. Ann.}   {\bf 29}  (1887), 331-338

\bibitem{Cap1bis-BR}
A. Capelli, Sur les op\'{e}rations dans la th\'{e}orie des formes alg\'{e}briques,
\textit{Math. Ann.}   {\bf 37}  (1890), 1-37



%
\bibitem{Cap2-BR}
A. Capelli, Sul sistema completo delle operazioni di polare
permutabili con ogni altra operazione di polare fra le stesse serie
di variabili, \textit{Rend. Regia Acc. Scienze Napoli} vol. VII
(1893), 29 - 38

%
\bibitem{Cap3-BR}
A. Capelli, Dell'impossibilit\`{a} di sizigie fra le operazioni
fondamentali permutabili con ogni altra operazione di polare fra le
stesse serie di variabili, \textit{Rend. Regia Acc. Scienze Napoli},
vol. VII  (1893), 155 - 162

%
\bibitem{Cap4-BR}
A. Capelli, {\it Lezioni sulla teoria delle forme algebriche,\/}
Pellerano, Napoli, 1902, available at \url{< https://archive.org/details/lezionisullateo00capegoog>}.

%
\bibitem{CL-BR}
W. Y. C. Chen and J. D. Louck, The factorial Schur function, {\it J. Math.
Phys.} {\bf 34} (1993), 4144--4160

%
\bibitem{CW-BR}
S.-J. Cheng, W. Wang, Howe duality for Lie superalgebras, {\it
Compositio Math.\/} {\bf 128} (2001),  55--94

%
\bibitem{DEP-BR}
C. De Concini, D. Eisenbud, C. Procesi, Young diagrams and
determinantal varieties, {\it Invent. Math.\/} {\bf 56} (1980),
129--165.

%
\bibitem{DKR-BR}
J. D\' esarm\' enien, J. P. S. Kung,  G.-C. Rota, Invariant theory,
Young bitableaux and combinatorics, {\it Adv. Math.\/} {\bf 27}
(1978), 63--92

%
\bibitem{DIX-BR}
J. Dixmier, {\it Enveloping algebras,\/} Graduate Studies in Mathematics 11, American Mathematical
Society, Providence, RI, 1996.

%
\bibitem{drs-BR}
P. Doubilet, G.-C. Rota, J. A. Stein, On the foundations of
combinatorial theory IX. Combinatorial methods in invariant theory,
{\it Studies in Appl. Math.\/} {\bf 53} (1974), 185--216

%
\bibitem{GH-BR}
I. P. Goulden and A. M. Hamel, Shift operators and factorial symmetric
functions, {\it J. Comb. Theor. A.} {\bf 69} (1995), 51--60


%
\bibitem{goulden-BR}
I. P. Goulden and D. M. Jackson, Immanants, Schur functions
and the MacMahon Master Theorem,  {\it Proc. Amer.
Math. Soc.\/} {\bf 115} (1992), 605--612

%
\bibitem{rota-BR}
F. D. Grosshans, G.-C. Rota and J. A. Stein,
\textit{Invariant Theory and Superalgebras}, AMS, 1987



\bibitem{Howe-BR}
R. Howe, Remarks on classical invariant theory, {\it Trans. Amer.
Math. Soc.\/} {\bf 313} (1989), 539--570

%
\bibitem{HU-BR}
R. Howe, T. Umeda, The Capelli identity, the double commutant
theorem, and multiplicity-free actions, {\it Math. Ann.} 290
(1991), 565-619

%
\bibitem{JK}
G. James, A. Kerber, {\it The Representation theory of the
symmetric group,\/} Encyclopedia of Mathematics and Its
Applications, vol.~16, Addison--Wesley, Reading, MA, 1981

%
\bibitem{KAC1-BR}
V. Kac,  Lie Superalgebras, {\it Adv. Math.\/} {\bf 26} (1977),
8--96

%
\bibitem{KostantSahi1-BR}
B Kostant and S. Sahi,  The Capelli identity, Tube Domains and the Generalized Laplace Transform, {\it Adv. Math.\/} {\bf 87} (1991),
71--92

%
\bibitem{KostantSahi2-BR}
B Kostant and S. Sahi,  Jordan algebras and Capelli identities,
{\it Invent. Math.\/} {\bf 112} (1993),
657--664

%
\bibitem{Koszul-BR}
J.-L. Koszul,  Les algèbres de Lie graduées de type sl(n,1) et
l'opérateur de A. Capelli, {\it C. R. Acad. Sci. Paris Sér. I
Math.\/}  {\bf 292} (1981), no. 2, 139-141



%
\bibitem{Little1-BR}
D. E. Littlewood,  A.R. Richardson, Group characters and algebras,
{\it Philosophical Transactions of the Royal Society A \/} {\bf{ 233}}  (1934),  99--124


%
\bibitem{Little2-BR}
D. E. Littlewood,  \textit{The Theory of Group Characters and Matrix Representations of Groups}, 2nd ed.,
 Oxford Univ. Press, 1950 (reprinted by AMS, 2006)



%
\bibitem{Molev1-BR}
A.I. Molev, Yangians and their applications, in {\it Handbook of Algebra, vol. $3$} (M.Hazewinkel, Ed.), pp. $907-960$,
Elsevier, 2003

%
\bibitem{Molev-BR}
A.I. Molev, {\it Yangians and Classical Lie Algebras}, Mathematical Surveys and Monographs,
{\bf 143}, Amer. Math. Soc., Providence RI, 2007


%
\bibitem{MolevNazarov-BR}
A.I. Molev and M. Nazarov, Capelli identities for classical Lie algebras, {\it Math. Ann.}
313 (1999), 315-357



%
\bibitem{Nazarov-BR}
M. Nazarov, Quantum Berezinian and the classical Capelli identity, {\it Lett. Math. Phys.}
21 (1991), 123-131


%
\bibitem{Nazarov2-BR}
M. Nazarov, Yangians and Capelli identities, in {\it Kirillov's seminar on representation
theory} (G. I. Olshanski, Ed.), AMS Translations, Series 2, vol. 181 (1998), pp. 139-163



%
\bibitem{Okounkov-BR}
A. Okounkov, Quantum immanants and higher Capelli identities,  {\it Transformation
Groups} 1 (1996), 99-126

%
\bibitem{Okounkov1-BR}
A. Okounkov, Young basis, Wick formula, and higher Capelli identities,  {\it Intern. Math. Res. Notices}  (1996),
no. 17, 817--839

%
\bibitem{OkOlsh-BR}
A. Okounkov, G. I.  Olshanski, Shifted Schur functions, {\it Algebra
i Analiz\/} {\bf 9}(1997), no. 2, 73--146 (Russian); English
translation: {\it St. Petersburg Math. J.} {\bf 9} (1998), 239--300

%
\bibitem{OkVer-BR}
A. Okounkov,  A. Vershik, A New Approach to Representation Theory of Symmetric Groups,
{\it Selecta Mathematica} {\bf 2} (2005), 581--605

%
\bibitem{Olsh1-BR}
G. I. Olshanski, Extension of the algebra U(g) for infinite-dimensional classical
Lie algebras g, and the Yangians Y (gl(m)), {\it Soviet Math. Dokl.\/} {\bf 36}  (1988), 569--573.

%
\bibitem{Olsh3-BR}
G. I. Olshanski, Representations of infinite-dimensional classical groups, limits of enveloping
algebras, and Yangians, in {\it Topics in Representation Theory} (A. A. Kirillov, Ed.), Advances
in Soviet Math. {\bf 2}, AMS, Providence RI, 1991, pp. 1--66



%
\bibitem{Procesi-BR}
C. Procesi, \textit{Lie Groups. An approach through invariants and representations},
Universitext, Springer, 2007


%
\bibitem{Sahi1-BR}
S. Sahi, The Spectrum of Certain Invariant Differential Operators Associated to a
Hermitian Symmetric Space, in {\it Lie theory and Geometry: in honor of Bertram Kostant},
(J.-L. Brylinski, R.. Brylinski, V. Guillemin, V. Kac, Eds.), Progress in Mathematics, Vol. 123, pp. 569--576, Birkhauser, 1994


%
\bibitem{Scheu-BR}
M. Scheunert, {\it The theory of Lie superalgebras: an
introduction,\/} Lecture Notes in Math., vol.~716, Springer Verlag,
New York, 1979

%
\bibitem{UmedaCent-BR}
T. Umeda, The Capelli identity one century after, in: {\it{Selected Papers on Harmonic Analysis,
Groups and Invariants}}, pp. 51-78, Amer. Math. Soc. Transl. Ser. 2, {\bf 183}, AMS, Providence, RI, 1998

%
\bibitem{Umeda-BR}
T. Umeda,  On the proof of the Capelli identities, \textit{Funkcialaj Ekvacioj} {\bf 51}  (2008), 1-15

%
\bibitem{UmedaHirai-BR}
T. Umeda,
On Turnbull identity for skew-symmetric matrices,
\textit{Proceedings of the Edinburgh Mathematical Society}
(Series 2), \textbf{43} (2000), 379-393.



%
\bibitem{Wall-BR}
A.H. Wallace, Invariant matrices and the Gordan--Capelli series,
{\it Proc. London Math. Soc.\/} {\bf 2} (1952), 98--127.


%
\bibitem{Weyl-BR}
H. Weyl, {\it The Classical Groups}, 2nd ed.,  Princeton University
Press, 1946


\end{thebibliography}
\end{document}